\documentclass[letterpaper, 10 pt, conference]{ieeeconf}  

\IEEEoverridecommandlockouts                              

\usepackage{graphics} 
\usepackage{subfig}
\usepackage{epsfig} 
\usepackage{mathptmx} 
\usepackage{times} 
\usepackage{amsmath} 
\usepackage{amssymb}  
\usepackage{multirow}
\usepackage{algorithm}
\usepackage{algorithmic}

\newtheorem{theorem}{Theorem}

\newtheorem{lemma}[theorem]{Lemma}

\newtheorem{definition}{Definition}

\newtheorem{remark}[theorem]{Remark}
\newtheorem{assumption}[theorem]{Assumption}

\newcommand{\BE}{\mathbb{E}}

\,
\,

\newcommand{\sr}{\mathbf r}
\newcommand{\x}{\mathbf x}
\newcommand{\y}{\mathbf y}
\newcommand{\z}{\mathbf z}
\newcommand{\su}{\mathbf u}
\newcommand{\sv}{\mathbf v}
\newcommand{\g}{\mathbf g}

\newcommand{\argmin}{\mathop{\rm argmin}}

\newcommand{\prox}{\textnormal{prox}}

\newcommand{\dom}{\textnormal{dom}}

\newcommand{\br}{\mathbb{R}}

\newcommand{\ba}{\begin{array}}
\newcommand{\ea}{\end{array}}
\newcommand{\ACal}{\mathcal{A}}
\newcommand{\BCal}{\mathcal{B}}

\title{\LARGE \bf Improved Sample Complexity for Stochastic Compositional Variance Reduced Gradient}

\author{Tianyi Lin, Chenyou Fan, Mengdi Wang and Michael I. Jordan\thanks{Tianyi Lin is with Department of IEOR, UC Berkeley. Chenyou Fan is with Google. Mengdi Wang is with Department of ORFE, Princeton University. Michael I. Jordan is with Department of EECS and Statistics, UC Berkeley. Email: {\tt\small \{darren{\_}lin@, jordan@cs\}berkeley.edu, fanchenyou@gmail.com, mengdiw@princeton.edu}}
}

\begin{document}

\maketitle
\thispagestyle{empty}
\pagestyle{empty}

\begin{abstract}
Convex composition optimization is an emerging topic that covers a wide range of applications arising from stochastic optimal control, reinforcement learning and multi-stage stochastic programming. Existing algorithms suffer from unsatisfactory sample complexity and practical issues since they ignore the convexity structure in the algorithmic design. In this paper, we develop a new stochastic compositional variance-reduced gradient algorithm with the sample complexity of $O((m+n)\log(1/\epsilon)+1/\epsilon^3)$ where $m+n$ is the total number of samples. Our algorithm is near-optimal as the dependence on $m+n$ is optimal up to a logarithmic factor. Experimental results on real-world datasets demonstrate the effectiveness and efficiency of the new algorithm. 
\end{abstract}

\section{Introduction}\label{sec:introduction}
We consider the convex composition optimization model: 
\begin{equation}\label{prob:N}
\min_{\x \in X} \ \Phi(\x) = F(\x) + r(\x), 
\end{equation}
where $F(\x) = \frac{1}{n}\sum_{i=1}^n f_i(\frac{1}{m}\sum_{j=1}^m g_j(\x))$ is convex with $f_i:\br^k \rightarrow \br$ and $g_j: X \rightarrow \br^k$ being continuously differentiable but possibly nonconvex, $r$ is an extended real-valued closed convex function on $\br^d$ and $X \subseteq \br^d$ is a convex set. Problem~\eqref{prob:N} is highly structured but also quite general; indeed, various special cases have been studied in a host of applications, including stochastic optimal control~\cite{Bertsekas-1995-Neuro}, multi-stage stochastic programming~\cite{Shapiro-2009-Lectures} and portfolio mean-variance optimization~\cite{Ravikumar-2009-Sparse}. A typical example is on-policy reinforcement learning~\cite{Sutton-1998-Reinforcement}: given a controllable Markov chain with states $\{1, \ldots, S\}$, a policy space $X$, a discount factor $\gamma \in (0, 1)$, a collection of transition probability matrices $\{P_j^\pi\}_{j=1}^m$ and $\{\sr_j^\pi\}_{j=1}^m$ from a simulator and a loss function $L$, the problem of solving the Bellman equation in a black-box simulation environment resorts to solving an optimization model:
\begin{equation*}
\min_{\x\in X} \ L\left( \frac{1}{m}\left[\sum_{j=1}^m \left(I-\gamma P_j^\pi\right) \x - \sr_j^\pi\right]\right),
\end{equation*} 
This is equivalent to problem~\eqref{prob:N} with $n=1$, $f_1 = \ell$ and $g_j(\x) = (I-\gamma P_j^\pi)\x - \sr_j^\pi$ for all $j$. 

Despite the popularity, problem~\eqref{prob:N} is more computationally challenging than its noncompositional counterpart; i.e., problem~\eqref{prob:N} with $g_j(\x) = \x$. Both stochastic gradient descent (SGD) and stochastic variance reduced gradient (SVRG)~\cite{Allen-2016-Improved} deteriorate because they suffer from very high computational burden for computing $\frac{1}{m}\sum_{j=1}^m g_j(\x)$ at each iteration. In response to this, a number of efficient variants of compositional SGD (SCGD)~\cite{Wang-2017-Stochastic, Wang-2017-Accelerating} and compositional SVRG (SCVRG)~\cite{Lian-2017-Finite, Yu-2017-Fast, Huo-2018-Accelerated} have been developed for solving problem~\eqref{prob:N}. The current state-of-the-art sample complexity in terms of optimality gap for convex composition optimization is $O((m+n)^{2/3}/\epsilon^2)$ achieved by the VRSC-PG algorithm~\cite{Huo-2018-Accelerated}; see Table~\ref{Tab:SCGD}. 

There has been abundant recent work on compositional SVRG for solving compositional optimization problems with strongly convex objectives~\cite{Lian-2017-Finite, Yu-2017-Fast} and nonconvex objectives~\cite{Huo-2018-Accelerated}. However, there has been relatively little investigation of convex (but not strongly convex) objectives which indeed cover various application problems in high-dimensional statistical learning~\cite{Wainwright-2019-High}. In particular, the following important case has been neglected: the composition function $F$ is convex while some of the composition terms are nonconvex. Perhaps counterintuitively, it has been observed in that the average of loss functions can be convex even if some loss functions are nonconvex~\cite{Shalev-2016-Sdca}. This occurs when some of the loss functions are strongly convex. In the SVRG setting, Allen-Zhu~\cite{Allen-2016-Improved} has shown that the optimal dependence of the sample complexity on $n$ can be established when the objective is convex. What remains unknown is whether this conclusion holds true in the compositional setting. Thus, it is natural to ask: 
\begin{quote}
\textsf{\emph{Can an algorithm achieve optimal dependence on $m+n$ for convex composition optimization?}}
\end{quote}
In this paper, we provide an affirmative answer to this question by developing an algorithm with a sample complexity of $O((m+n)\log(1/\epsilon)+1/\epsilon^3)$ in terms of the optimality gap. The algorithm is intuitive and has an efficient implementation. The proof technique is new and of independent interest.

The rest of the paper is organized as follows. In Section~\ref{Sec:prelim}, we discuss the notation, definitions and assumptions. In Section~\ref{Sec:algorithm}, we describe the algorithm and discuss its differences from the algorithms in Table~\ref{Tab:SCGD}. In Section~\ref{Sec:sample}, we provide the sample complexity analysis.  We place some technical detail in the Appendix. In Section~\ref{Sec:results}, we present numerical results on real-world datasets.
\begin{table}[!t]
\caption{Sample Complexity of Candidate Algorithms.}\label{Tab:SCGD}\vspace{-.5em}\scriptsize
\centering
\begin{tabular}{|l|l|c|c|} \hline 
 & Sample Complexity & $r \neq 0$? & Shrinking $\eta$? \\ \hline  
AGD~\cite{Nesterov-2013-Introductory} & $O((m+n)/\sqrt{\epsilon})$ & Yes & No \\ \hline
SCGD~\cite{Wang-2017-Stochastic} & $O(1/\epsilon^4)$ & No & Yes \\ \hline
ASC-PG \cite{Wang-2017-Accelerating} & $O(1/\epsilon^{3.5})$ & Yes & Yes \\ \hline
VRSC-PG \cite{Huo-2018-Accelerated} & $O((m+n)^{2/3}/\epsilon^2)$ & Yes & No \\ \hline
\textbf{This paper} & $O((m+n)\log(1/\epsilon)+1/\epsilon^3)$ & Yes & No \\ \hline
\end{tabular}
\end{table}

\section{Preliminaries}\label{Sec:prelim}
\textbf{Notation.} Vectors are denoted by bold lower case letters. $\left\|\cdot\right\|$ denotes the $\ell_2$-norm and the matrix spectral norm. For two sequences $a_t, b_t \geq 0$, we write $a_t=O(b_t)$ if $a_t \leq N b_t$ for some constant $N>0$. We denote the gradient\footnote{We refer $\nabla f(g(\x))$ to the gradient of $f(\y)$ at $\y=g(\x)$, not the gradient of $F(\x)=f(g(\x))$ at $\x$.} of $F$ at $\x$ as $\nabla F(\x)=[\partial g(\x)]^\top\nabla f(g(\x))$, where $\partial g$ is the Jacobian matrix of $g$. $\BE\left[\cdot\mid\zeta\right]$ refers to expectation conditioned on $\zeta$. 

\textbf{Goals in stochastic composition optimization:} We wish to find a point that globally minimizes the objective $\Phi$.
\begin{definition}
$\x^* \in X$ is an optimal solution to problem~\eqref{prob:N} if $\Phi(\x) - \Phi(\x^*) \geq 0$ for $\forall \x \in X$.
\end{definition}
In general, finding such global optimal solution is NP-hard~\cite{Murty-1987-Some} but standard for convex optimization. To this end, we make the following assumption throughout this paper.  
\begin{assumption}\label{Assumption:Objective-Convex-Main}
The objective $F$ and $r$ are both convex:
\begin{align*}
F(\x) - F(\y) - \left\langle \nabla F(\y), \x-\y\right\rangle & \geq 0, \quad \forall \x, \y \in X, \\ 
r(\x) - r(\y) - \left\langle\xi, \x-\y\right\rangle & \geq 0, \quad \forall \x, \y \in X, 
\end{align*}
where $\xi\in\partial r(\y)$ is a subgradient of $r$ at $\y$ and the constraint set $X$ is bounded. The proximal mapping of $r$, defined by
\begin{eqnarray*}
\prox_r^\eta(\x) = \argmin_{\y \in X}\left\{r(\y) + \frac{1}{2\eta}\left\|\y - \x\right\|^2\right\}, 
\end{eqnarray*}
can be efficiently computed for any $\x$ and any $\eta>0$. 
\end{assumption}
Assumption~\ref{Assumption:Objective-Convex-Main} is not restrictive since $r$ is an indicator function of a convex and bounded set in many applications~\cite{Parikh-2014-Proximal}. Furthermore, a minimal set of conditions that have become standard in the literature~\cite{Lian-2017-Finite,Yu-2017-Fast,Huo-2018-Accelerated} are as follows. 
\begin{assumption}\label{Assumption:Smooth-Gradient-Jacobian-Main} $f_i$ is $L_f$-Lipschitz and $\ell_f$-gradient Lipschitz on $X$, i.e., 
\begin{align*}
\left\| f_i(\x) - f_i(\y)\right\| & \leq L_f \left\|\x-\y\right\|, \quad \forall\x, \y \in X \\
\left\| \nabla f_i(\x) - \nabla f_i(\y)\right\| & \leq \ell_f \left\|\x-\y\right\|, \quad \forall\x, \y \in X.  
\end{align*}
Also, $g_j$ is $L_g$-Lipschitz and $\ell_g$-Jacobian Lipschitz on $X$: 
\begin{align*}
\left\| g_j(\x) - g_j(\y)\right\| & \leq L_g \left\|\x-\y\right\|, \quad \forall\x, \y \in X  \\
\left\| \partial g_j(\x) - \partial g_j(\y)\right\| & \leq \ell_g \left\|\x-\y\right\|, \quad \forall\x, \y \in X.   
\end{align*}
\end{assumption}
Given that $\dom(\Phi)$ is bounded, Assumption~\ref{Assumption:Smooth-Gradient-Jacobian-Main} is not restrictive but requires that $f_i$ and $g_j$ are continuously differentiable. Assumption~\ref{Assumption:Smooth-Gradient-Jacobian-Main} also implies that 
\[\|\nabla F(\x) - \nabla F(\y)\| \leq \ell\|\x - \y\|, \quad \ell = L_f\ell_g + L_g^2\ell_f. \]
We define an $\epsilon$-optimal solution in stochastic optimization via a relaxation of the global optimality condition.
\begin{definition}\label{Definition:eps-optimal}
$\x \in X$ is an $\epsilon$-optimal solution to problem~\eqref{prob:N} if $\Phi(\x) \leq \Phi(\x^*) + \epsilon$ where $\x^*$ is an optimal solution. 
\end{definition}
Throughout this paper, the algorithm efficiency is quantified by \textit{sample complexity}, i.e., the number of sampled $g_j(\x)$ or $\partial g_j(\x)$ for $\x \in \br^d$, or sampled $f_i(\y)$ or $\nabla f_i(\y)$ for $\y \in \br^k$, required to achieve an $\epsilon$-optimal solution in Definition~\ref{Definition:eps-optimal}.

With these definitions in mind, we ask if an algorithm can achieve a near-optimal sample complexity for convex composition optimization. 

\section{Algorithm}\label{Sec:algorithm}
In this section, we present our algorithm and discuss its differences from the algorithms in Table~\ref{Tab:SCGD}. We first choose a reference point $\tilde{\x}^s$ with $\tilde{\g}^{s+1}=g(\tilde{\x}^s)$, $\tilde{\z}^{s+1}=\partial g(\tilde{\x}^s)$, $\tilde{\sv}^{s+1}=[\tilde{\z}^{s+1}]^\top\nabla f(\tilde{\g}^{s+1})$ and 
{\small \begin{align}
\g_t^{s+1} & = \tilde{\g}^{s+1} + \frac{1}{a}\sum_{j_t \in \ACal_t} \left(g_{j_t}(\x_t^{s+1}) - g_{j_t}(\tilde{\x}^s)\right), \label{Update:SCVRG-g} \\
\z_t^{s+1} & = \tilde{\z}^{s+1} + \frac{1}{a}\sum_{j_t \in \ACal_t} \left(\partial g_{j_t}(\x_t^{s+1}) - \partial g_{j_t}(\tilde{\x}^s)\right). \label{Update:SCVRG-z}  \\
\sv_t^{s+1} & = \tilde{\sv}^{s+1} + \frac{1}{b}\sum_{i_t \in \BCal_t} \left([\z_t^{s+1}]^\top\nabla f_{i_t}(\g_t^{s+1}) - [\tilde{\z}^{s+1}]^\top\nabla f_{i_t}(\tilde{\g}^{s+1})\right). \label{Update:SCVRG-v} 
\end{align}}
Then we update $\x_{t+1}^{s+1}$ with a step size $\eta_{t+1}^{s+1}$:
\begin{equation}\label{Update:SCVRG-main}
\x_{t+1}^{s+1} = \prox_r^{\eta_{t+1}^{s+1}}\left(\x_t^{s+1} - \eta_{t+1}^{s+1}\sv_t^{s+1}\right). 
\end{equation}
and choose the reference point for the next epoch as the averaging iterates. Computing a full gradient vector and a full Jacobian matrix once requires $m+n$ samples and the sample complexity for the $s$th epoch is $m+n+k_s\left(a+b\right)$. 
\begin{algorithm}[!t]
\begin{algorithmic}
\STATE \textbf{Input:} $\tilde{\x}^0{=}\x_{k_0}^0{=}\x^0\in\br^d$, first epoch length $k_0$, step size $\eta>0$, the number of epochs $S$ and sample size $a, b > 0$.  
\STATE \textbf{Initialization:} $l=0$ and $T=k_02^S - k_0$.  
\FOR{$s=0,1,\ldots, S$}
\STATE $\x_0^{s+1} = \x_{k_s}^s$. 
\STATE $\tilde{\g}^{s+1} = g(\tilde{\x}^s)$. 
\STATE $\tilde{\z}^{s+1} = \partial g(\tilde{\x}^s)$. 
\STATE $\tilde{\sv}^{s+1} = \left[\tilde{\z}^{s+1}\right]^\top\nabla f(\tilde{\g}^{s+1})$. 
\STATE $k_{s+1} = k_0 2^{s+1}$. 
\FOR{$t = 0, 1, \ldots, k_{s+1} - 1$}
\STATE Uniformly sample with replacement a subset $\ACal_t \subseteq [m]$ with $|\ACal_t|=a$, and update by~\eqref{Update:SCVRG-g} and~\eqref{Update:SCVRG-z}. 
\STATE Uniformly sample with replacement a subset $\BCal_t \subseteq [n]$ with $|\BCal_t|=b$, and update by~\eqref{Update:SCVRG-v}. . 
\STATE Set $l=l+1$ and $\eta_{t+1}^{s+1}=\frac{\eta\sqrt{T}}{\sqrt{2T-l}}$, and update by~\eqref{Update:SCVRG-main}. 
\ENDFOR
\STATE $\tilde{\x}^{s+1}=\frac{1}{k_{s+1}}\sum_{t=0}^{k_{s+1}-1} \x_t^{s+1}$. 
\ENDFOR
\STATE \textbf{Output:} $\tilde{\x}^S$. 
\end{algorithmic} \caption{Stochastic Compositional Variance Reduced Gradient(SCVRG)} \label{Algorithm:SCVRG}
\end{algorithm}

Finally, we discuss the differences between Algorithm~\ref{Algorithm:SCVRG} and existing algorithms. Indeed, AGD is deterministic and SCGD/ASC-PG are compositional variants of SGD while Algorithm~\ref{Algorithm:SCVRG} is a compositional variant of SVRG. Algorithm~\ref{Algorithm:SCVRG} also differs from other existing compositional variants of SVRG~\cite{Lian-2017-Finite, Yu-2017-Fast, Huo-2018-Accelerated} since the number of inner loops is a constant in these algorithms while increasing in Algorithm~\ref{Algorithm:SCVRG}. Intuitively, Algorithm~\ref{Algorithm:SCVRG} gradually increases the number of inner loops as the iterate approaches the optimal set and yields an improved sample complexity by reducing the number of computations of the full gradient vector and Jacobian matrix. 
\begin{figure*}[!t] 
\subfloat[Book-to-Market]{%
\includegraphics[width=0.32\textwidth]{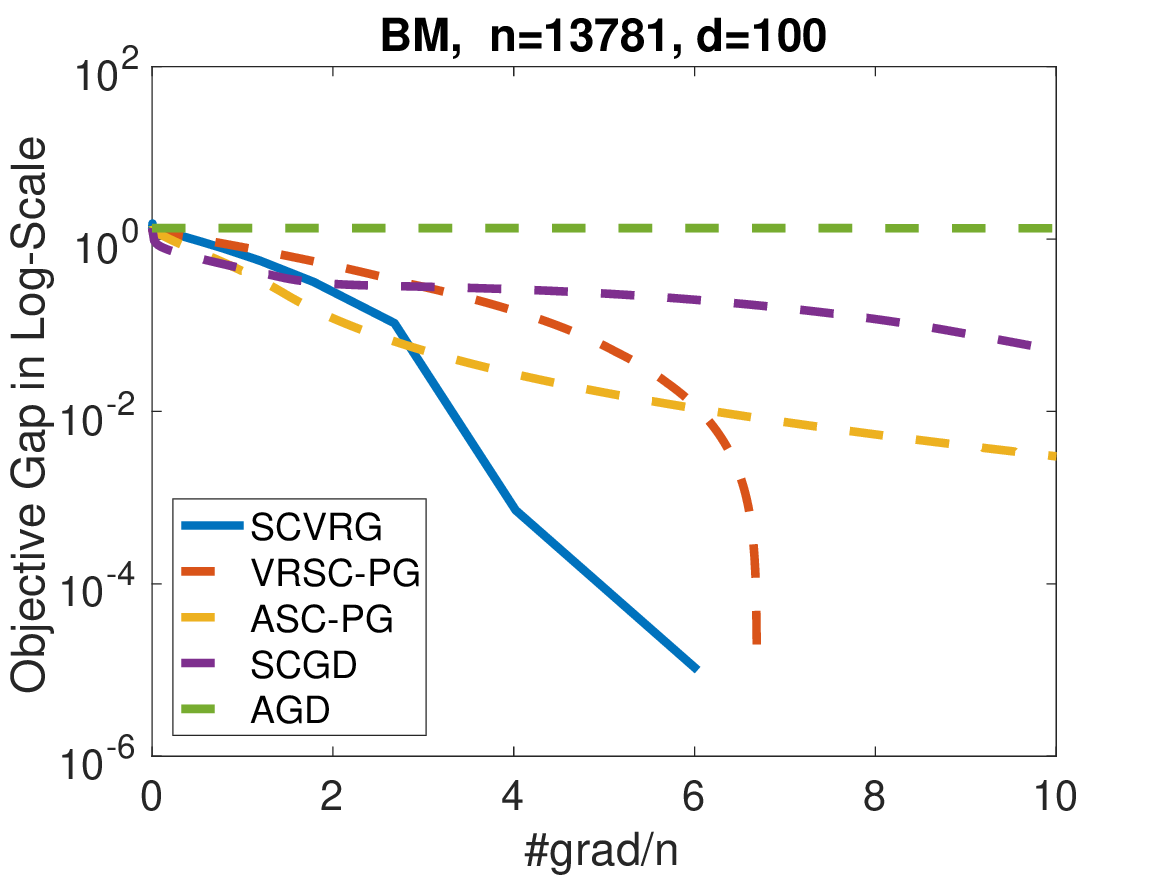}
} \quad 
\subfloat[Operating Profitability]{%
\includegraphics[width=0.32\textwidth]{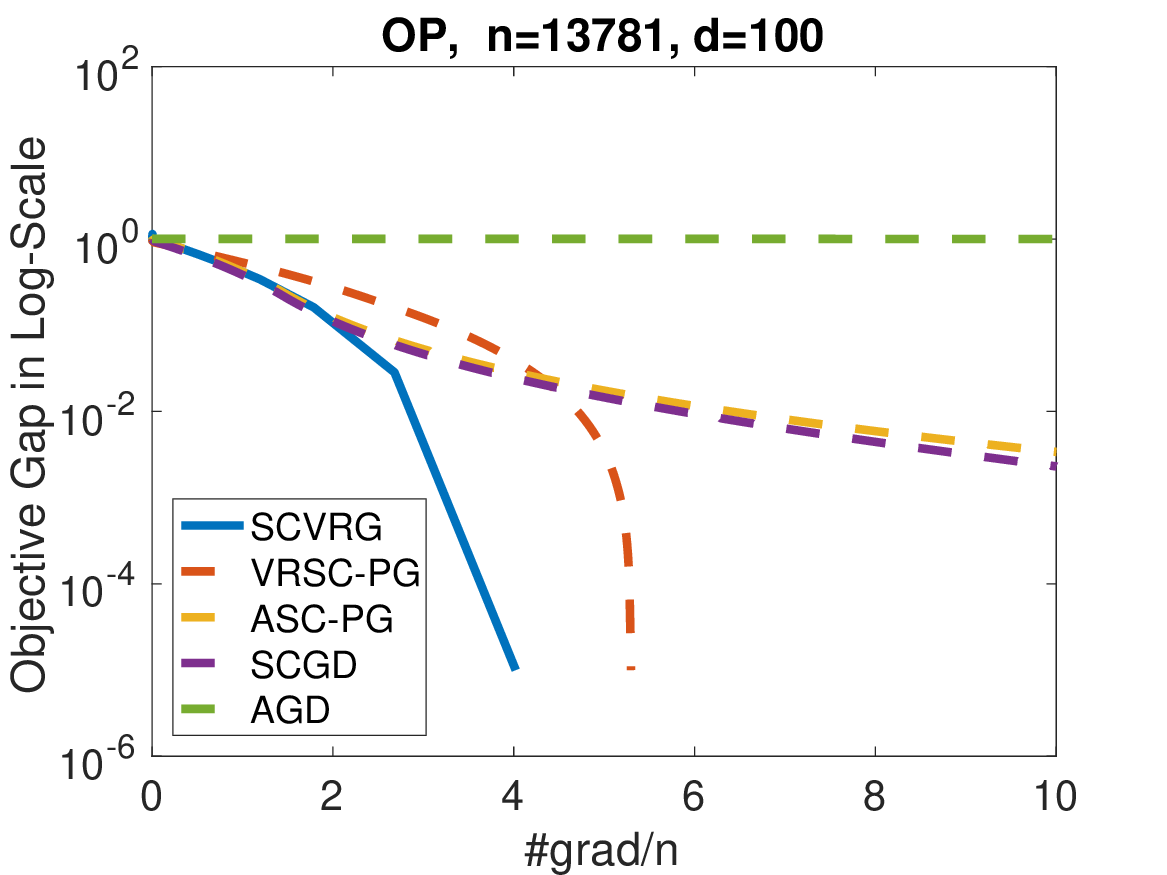}
} \quad 
\subfloat[Investment]{%
\includegraphics[width=0.32\textwidth]{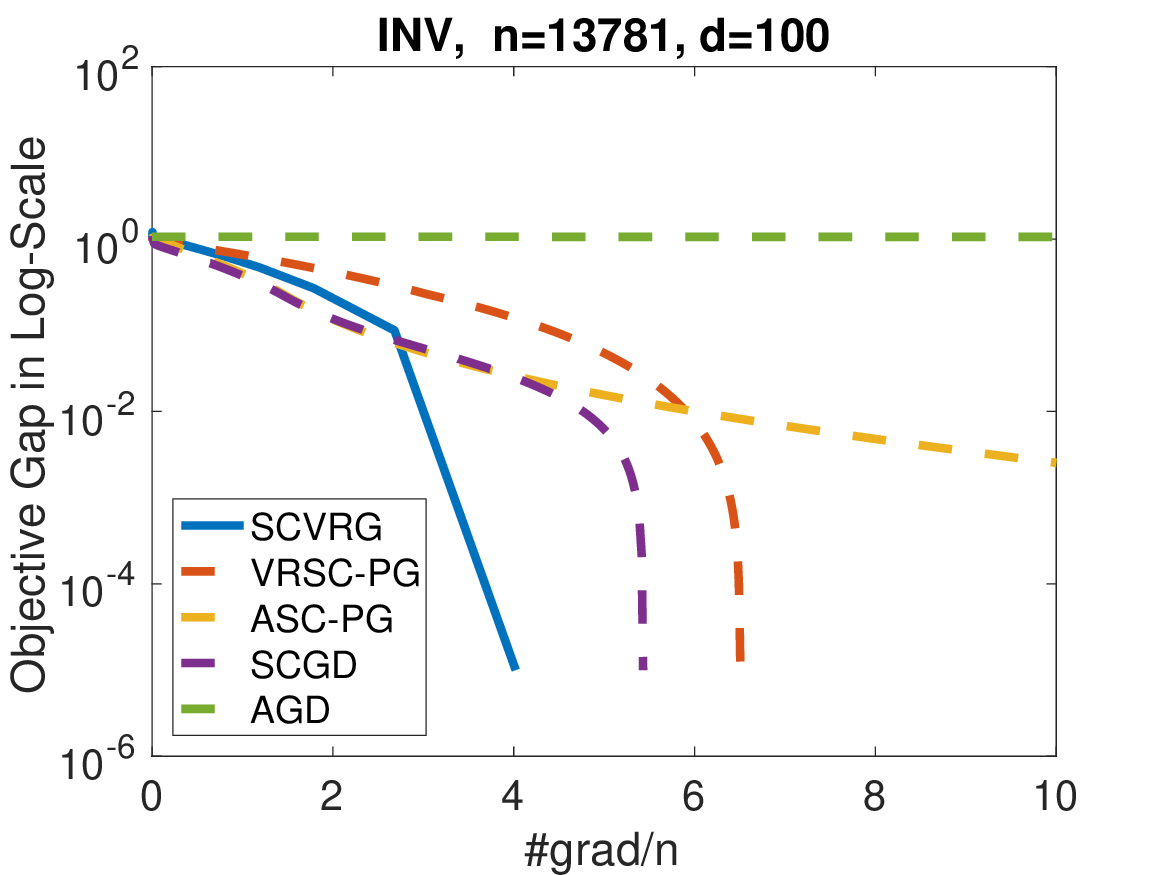}
} \quad 
\caption{The Performance of All Methods on three Large 100-Portfolio Datasets}\label{Fig:Large-Data}\vspace{-2em}
\end{figure*}

\section{Sample Complexity}\label{Sec:sample}
We provide the sample complexity of Algorithm~\ref{Algorithm:SCVRG}. 
\begin{theorem}\label{Theorem:SCVRG-Complexity-Main}
Let $\x^*\in X$ be an optimal solution, $\epsilon > 0$ be a small tolerance and let $\x^0 \in X$ satisfy $\|\x^0 - \x^*\| \leq D_\x$ and $\Phi(\x^0)-\Phi(\x^*) \leq D_\Phi$ where $D_\x, D_\Phi > 0$. The first epoch length $k_0 = 10$ and the number of epochs $S$ and $\eta > 0$ satisfy
\begin{equation*}
S = \left\lfloor\log_2\left(\frac{6D_\Phi + 15\ell D_\x^2}{\epsilon}\right) \right\rfloor + 1, \quad \eta = \frac{D_\x^2}{10D_\Phi + 25\ell D_\x^2}. 
\end{equation*}
The sample sizes $a, b > 0$ satisfy  
\begin{equation*}
a \geq \frac{1620\ell^2 D_\x^4}{\epsilon^2}, \quad b \geq \frac{810\ell^2 D_\x^4}{\epsilon^2}. 
\end{equation*}
Then the number of samples required to return $\widehat{\x} \in X$ such that $\BE[\Phi(\widehat{\x})] \leq \Phi(\x^*) + \epsilon$ is upper bounded by $O((m+n)\log(1/\epsilon) + 1/\epsilon^3)$. 
\end{theorem}
\begin{remark}
We discuss the results presented in Table~\ref{Tab:SCGD}. Indeed, the sample complexity of AGD, VRSC-PG and Algorithm~\ref{Algorithm:SCVRG} depend on $m+n$ while that of SCGD and ASC-PG is independent at the expense of worse dependence on $\epsilon$. The advantage of Algorithm~\ref{Algorithm:SCVRG} is that $m+n$ is \textbf{independent} of $\epsilon$ up to a logarithmic factor. Compared to VRSC-PG, Algorithm~\ref{Algorithm:SCVRG} is better if $m+n = O(\epsilon^{-\alpha})$ for $\alpha \in [1.5, 6]$. Compared to ASC-PG, Algorithm~\ref{Algorithm:SCVRG} is better than if $m+n = O(\epsilon^{-\alpha})$ for $\alpha \in [0, 3.5]$. This explains why Algorithm~\ref{Algorithm:SCVRG} performs better in many applications. For example, $\epsilon = 0.001$ is a common choice in practice to avoid over-fitting the training data, while many real-world learning problems have $m+n \in [10^5, 10^{10}]$ number of samples. 
\end{remark}
\begin{remark}
Algorithm~\ref{Algorithm:SCVRG} benefits from using adaptive $\eta_t^s$, which is crucial to the effectiveness and robustness of the algorithms when a huge number of iterations are required. 
\end{remark}

\section{Experiments} \label{Sec:results}
In this section, we present numerical results on real-world datasets\footnote{http://mba.tuck.dartmouth.edu/pages/faculty/ken.french/Data\_Library/}, including 3 large 100-portfolio datasets and 15 medium 25-portfolio datasets as shown in Table~\ref{Tab:Stat}.

Given $d$ assets and the reward vectors at $N$ time points, the goal of sparse mean-variance optimization~\cite{Ravikumar-2009-Sparse} is to maximize the return of the investment as well as to control the investment risk:
\begin{equation}\label{prob:SpMO}\small
\min\limits_{\x \in X} \ \frac{1}{N}\sum_{i=1}^N \left(\left\langle \sr_i, \x\right\rangle - \frac{1}{N}\sum_{i=1}^N \left\langle \sr_i, \x\right\rangle\right)^2 - \frac{1}{N}\sum_{i=1}^N \left\langle \sr_i, \x\right\rangle + \lambda\|\x\|_1.  
\end{equation}
Problem~\eqref{prob:SpMO} is exactly problem~\eqref{prob:N} with $m=n=N$, $f_i(\z, y) = \left(\left\langle \sr_i, \z\right\rangle + y\right)^2 - \left\langle \sr_i, \z\right\rangle$, $g_j(\x) = \left(\x^\top, -\left\langle \sr_j, \x\right\rangle\right)^\top$ for $\x,\z\in\br^d$ and $y\in\br$, $r(\cdot) = \|\cdot\|_1$ and $X$ is a bounded set. This problem satisfies Assumption~\ref{Assumption:Objective-Convex-Main} and~\ref{Assumption:Smooth-Gradient-Jacobian-Main} and serves as a typical example of the problem studied in this paper.   
\begin{table}[!t]
\caption{Statistics of CRSP Real Datasets}\label{Tab:Stat}\vspace{-.5em}
\scriptsize 
\centering
\begin{tabular}{|l|c|c|l|} \hline
Size & $N$ & $d$ & Problems \\ \hline
\multirow{3}{*}{Large} & \multirow{3}{*}{13781} & \multirow{3}{*}{100} & Book-to-Market (BM) \\
& & & Operating Profitability (OP) \\
& & & Investment (INV) \\ \hline
\multirow{3}{*}{Medium} & \multirow{3}{*}{7240} & \multirow{3}{*}{25} & BM: Asia, Europe, Global, Japan, America \\
& & & OP: Asia, Europe, Global, Japan, America \\
& & & INV: Asia, Europe, Global, Japan, America \\ \hline
\end{tabular}\vspace{-2.5em}
\end{table}

We denote Algorithm~\ref{Algorithm:SCVRG} as SCVRG and compare it with a line of existing algorithms. The implementations of baseline algorithms are provided by their authors with default parameters. We exclude the algorithm in \cite{Lian-2017-Finite, Yu-2017-Fast} since problem~\eqref{prob:SpMO} is neither smooth nor strongly convex. We set $A=B=5$ and $\eta = 0.01$, and choose $\lambda$ by cross validation. We use the number of samples used divided by $N$ for the $x$-axis and $\Phi(\x)-\Phi^*$ on a log scale for the $y$-axis\footnote{$\Phi^*$ is achieved by running Algorithm~\ref{Algorithm:SCVRG} until convergence to a highly accurate solution.}.

Figure~\ref{Fig:Large-Data} shows that SCVRG outperforms other algorithms on large datasets. In particular, SCVRG and VRSC-PG are more robust than SCGD and ASC-PG, demonstrating the superiority of variance reduction and constant stepsizes. AGD performs the worst because of the highest per-iteration sample complexity on large datasets. Figures~\ref{Fig:Book_to_Market}-\ref{Fig:Investment} show that SCVRG outperforms other algorithms on these datasets. This is consistent with the better complexity bound of SCVRG. Furthermore, VRSC-PG is less robust than SCVRG possibly because the number of inner loops in SCVRG is increasing which provides flexibility in handling different dataset sizes. Overall, SCVRG has the potential to be a benchmark algorithm for convex composition optimization. 
\begin{figure*}[!t] 
\subfloat[Asia]{
\includegraphics[width=0.2\textwidth]{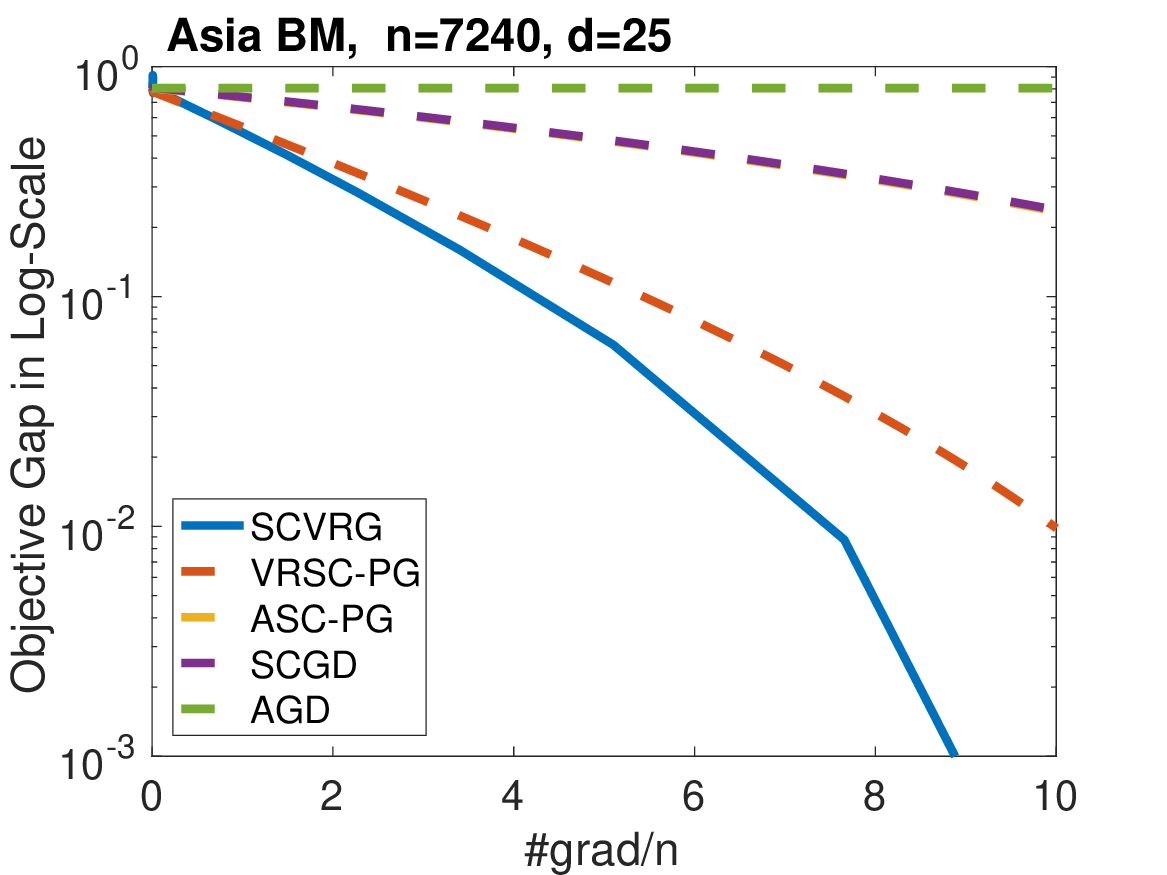}%
}
\subfloat[Europe]{%
\includegraphics[width=0.2\textwidth]{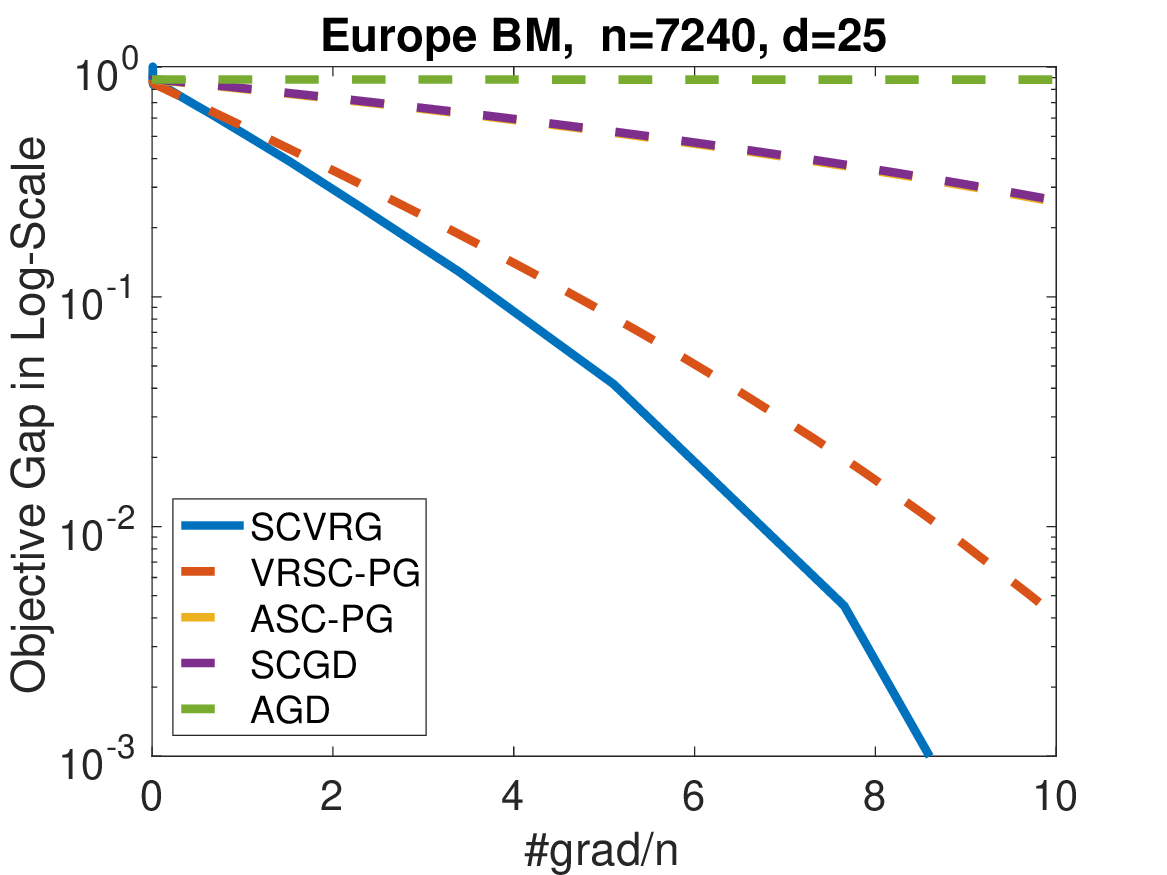}%
}%
\subfloat[Global]{%
\includegraphics[width=0.2\textwidth]{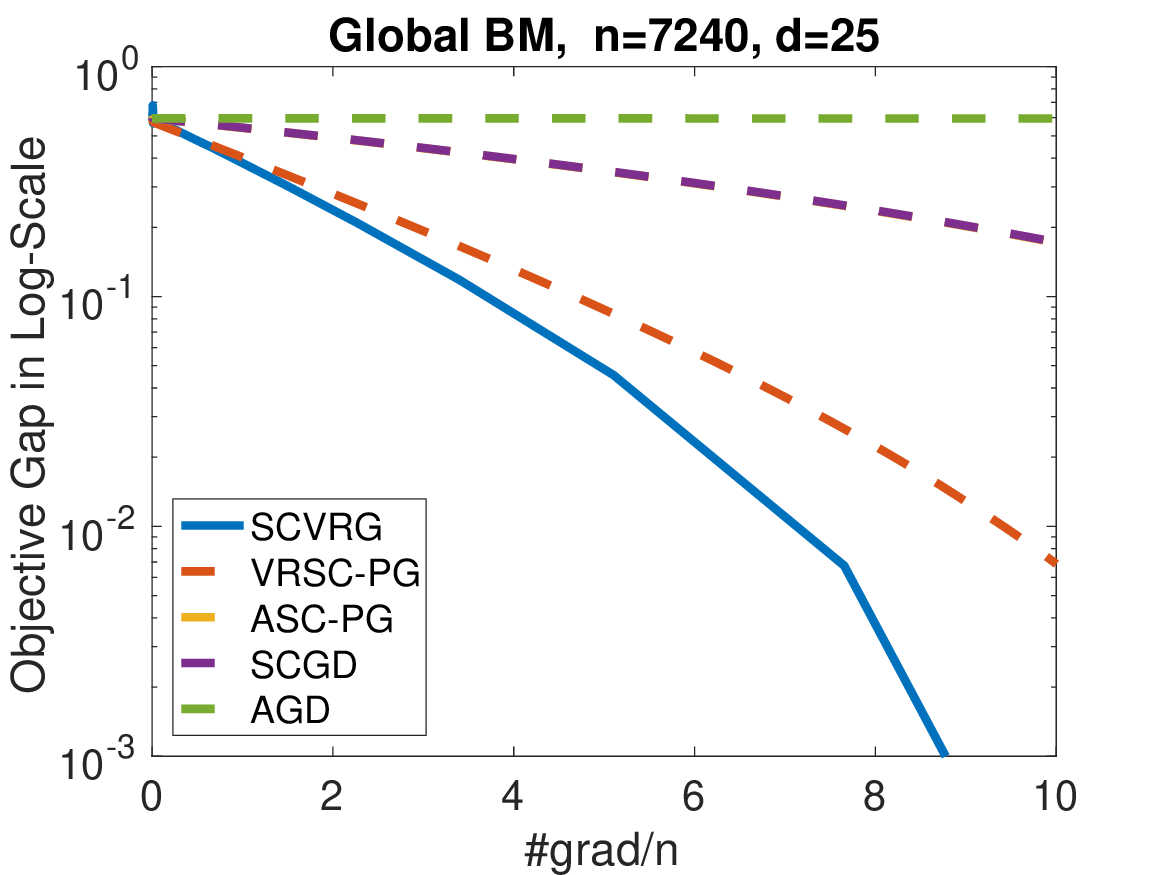}%
}%
\subfloat[Japan]{%
\includegraphics[width=0.2\textwidth]{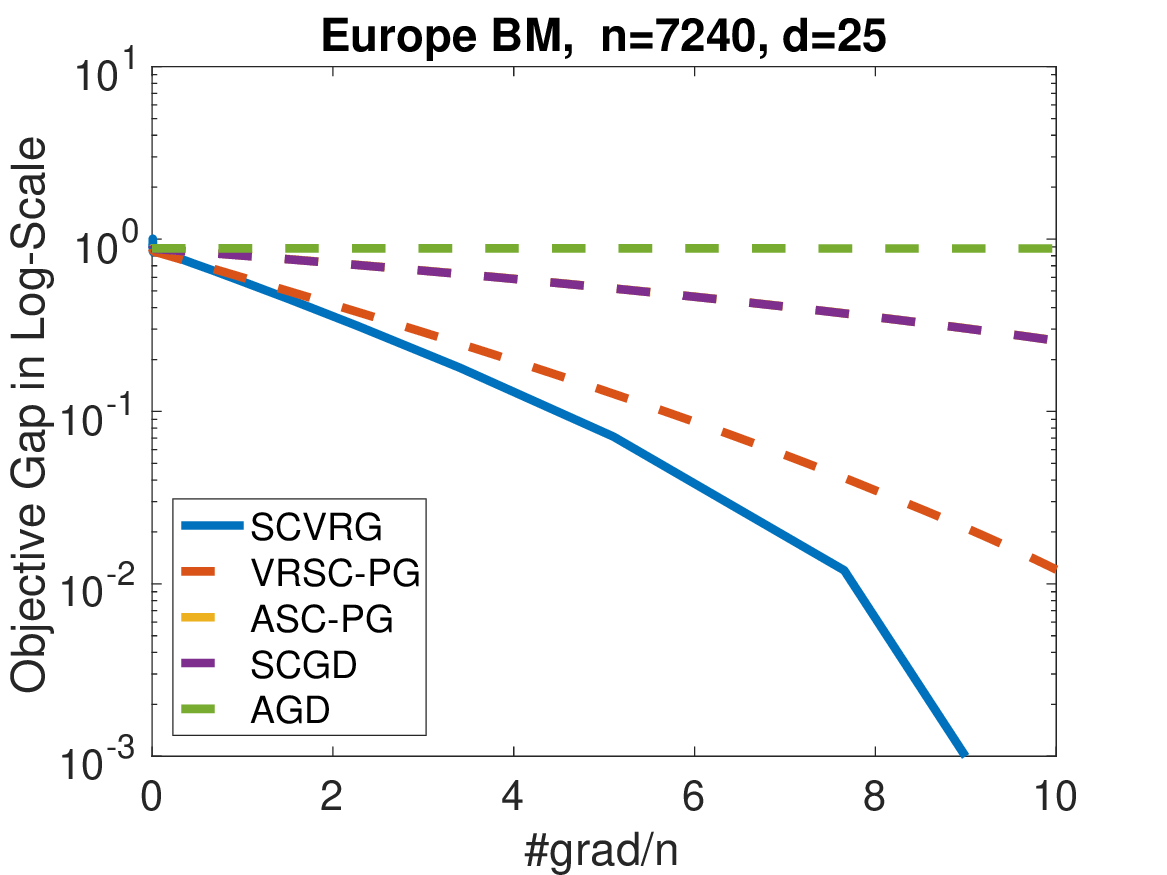}%
}%
\subfloat[America]{%
\includegraphics[width=0.2\textwidth]{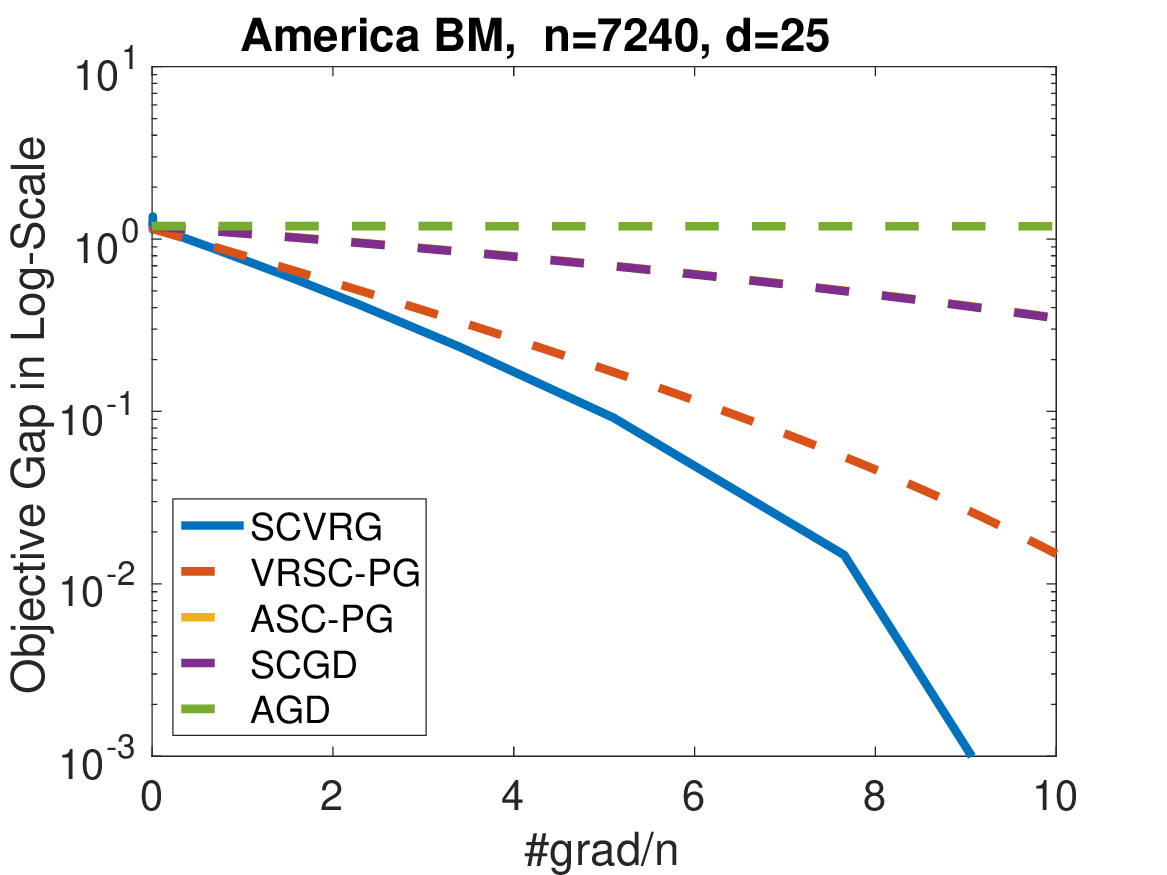}%
}%
\caption{The Performance of All Methods on five 25-Portfolio Book-to-Market Datasets}\label{Fig:Book_to_Market}\vspace{-2em}
\end{figure*}
\begin{figure*}[!t] 
\subfloat[Asia]{
\includegraphics[width=0.2\textwidth]{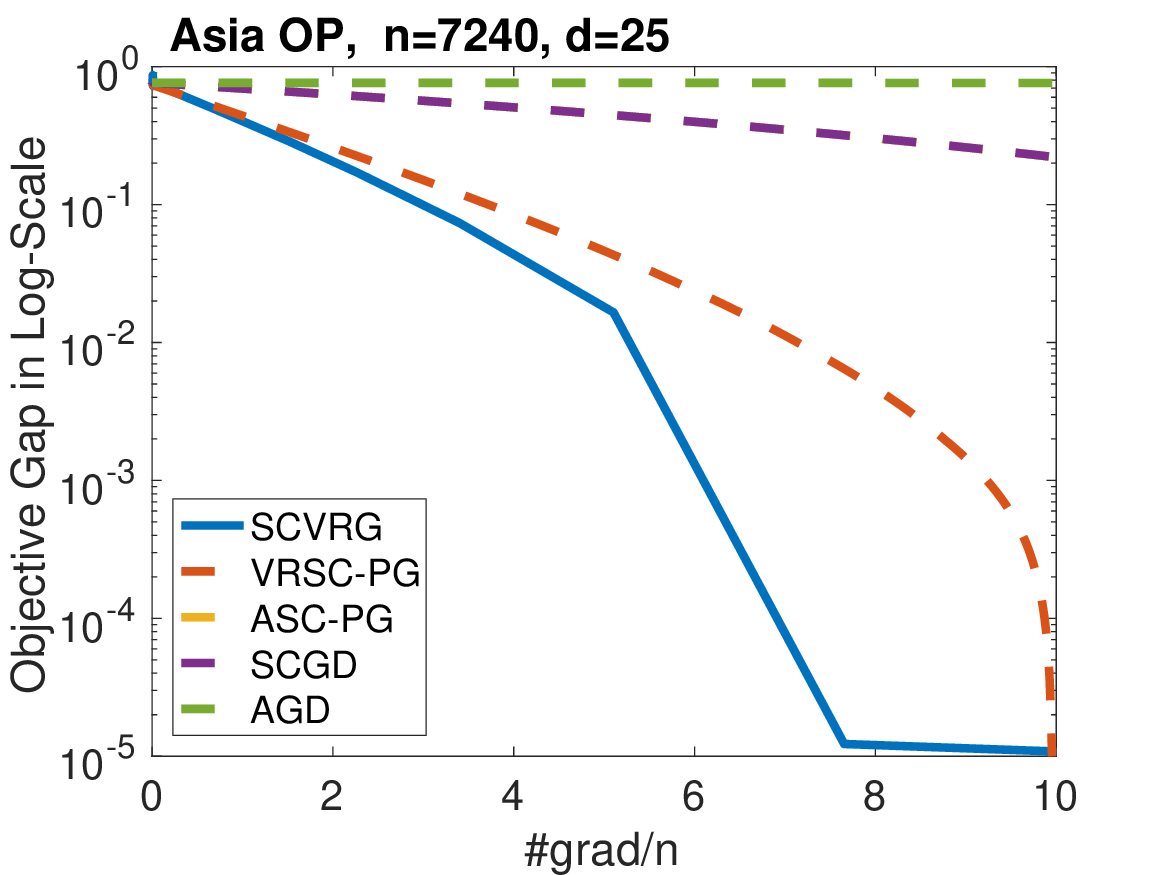}%
}
\subfloat[Europe]{%
\includegraphics[width=0.2\textwidth]{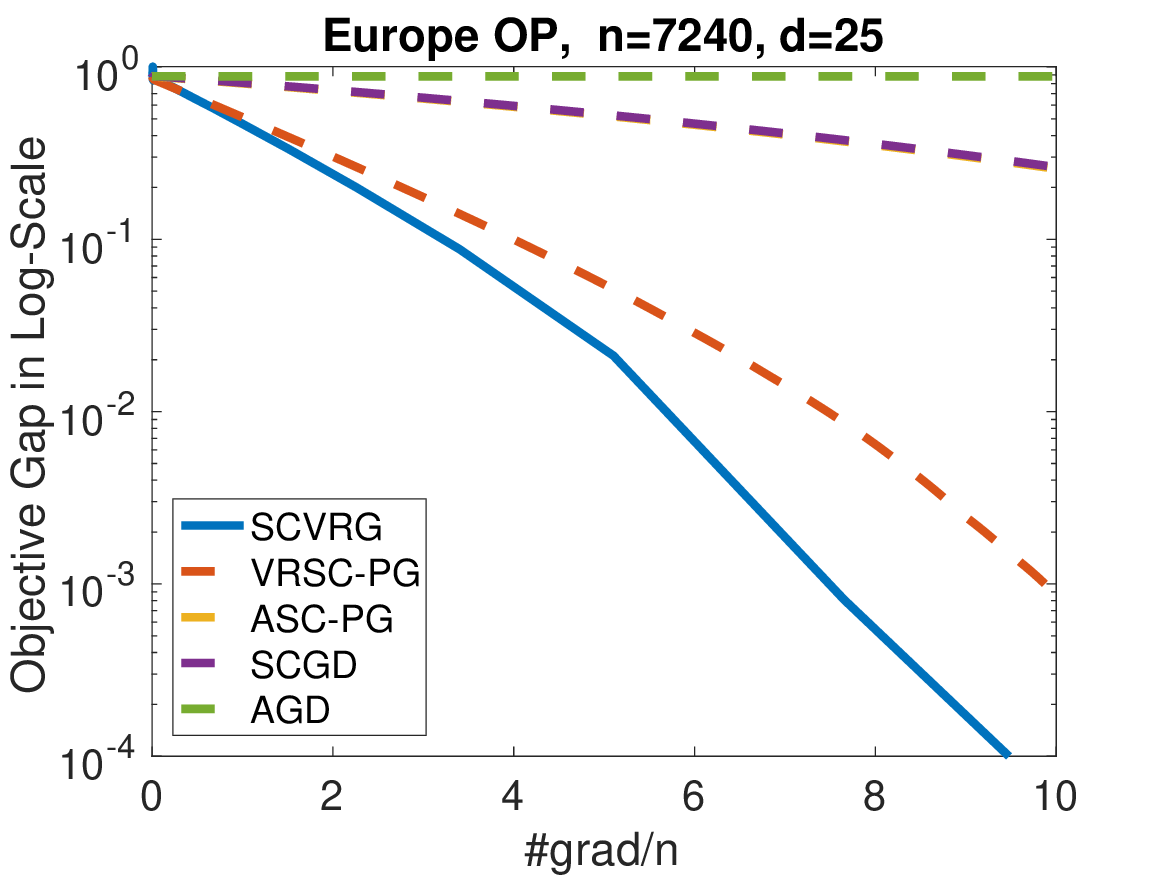}%
}%
\subfloat[Global]{%
\includegraphics[width=0.2\textwidth]{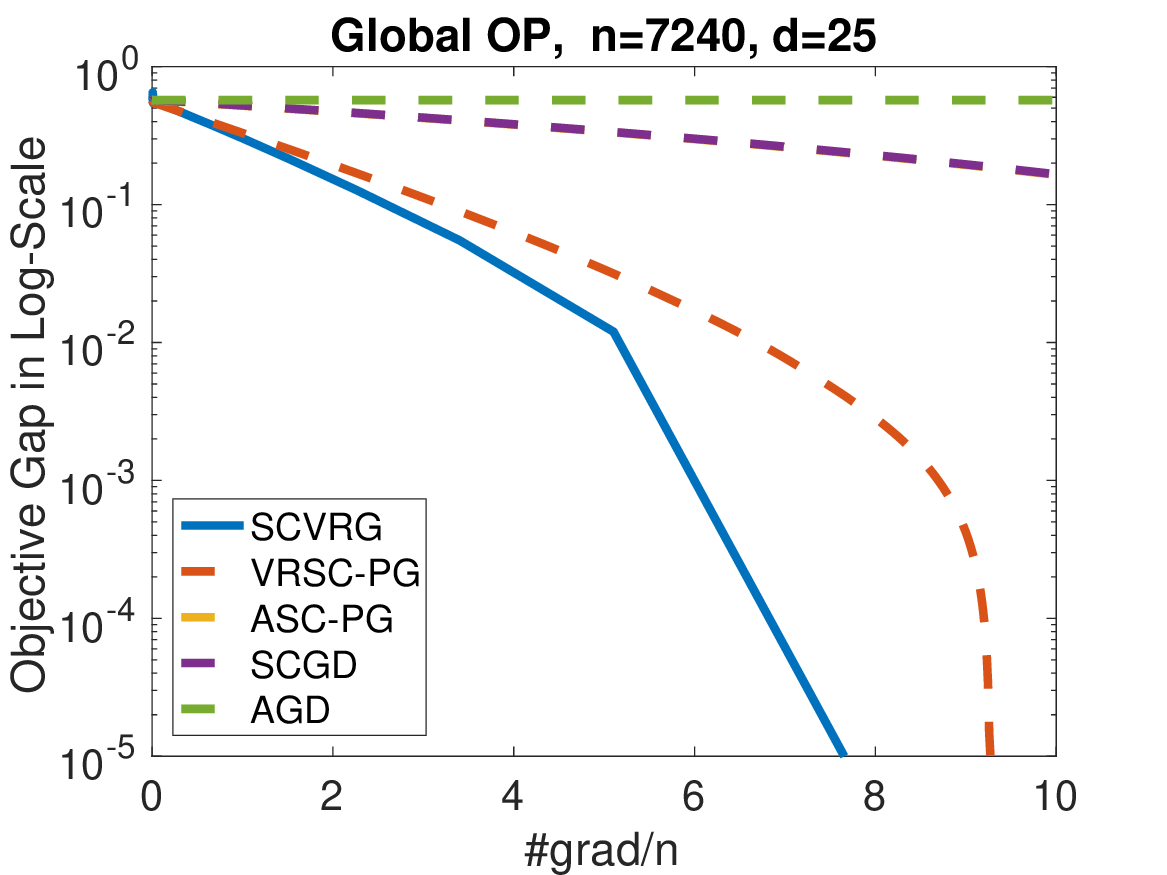}%
}%
\subfloat[Japan]{%
\includegraphics[width=0.2\textwidth]{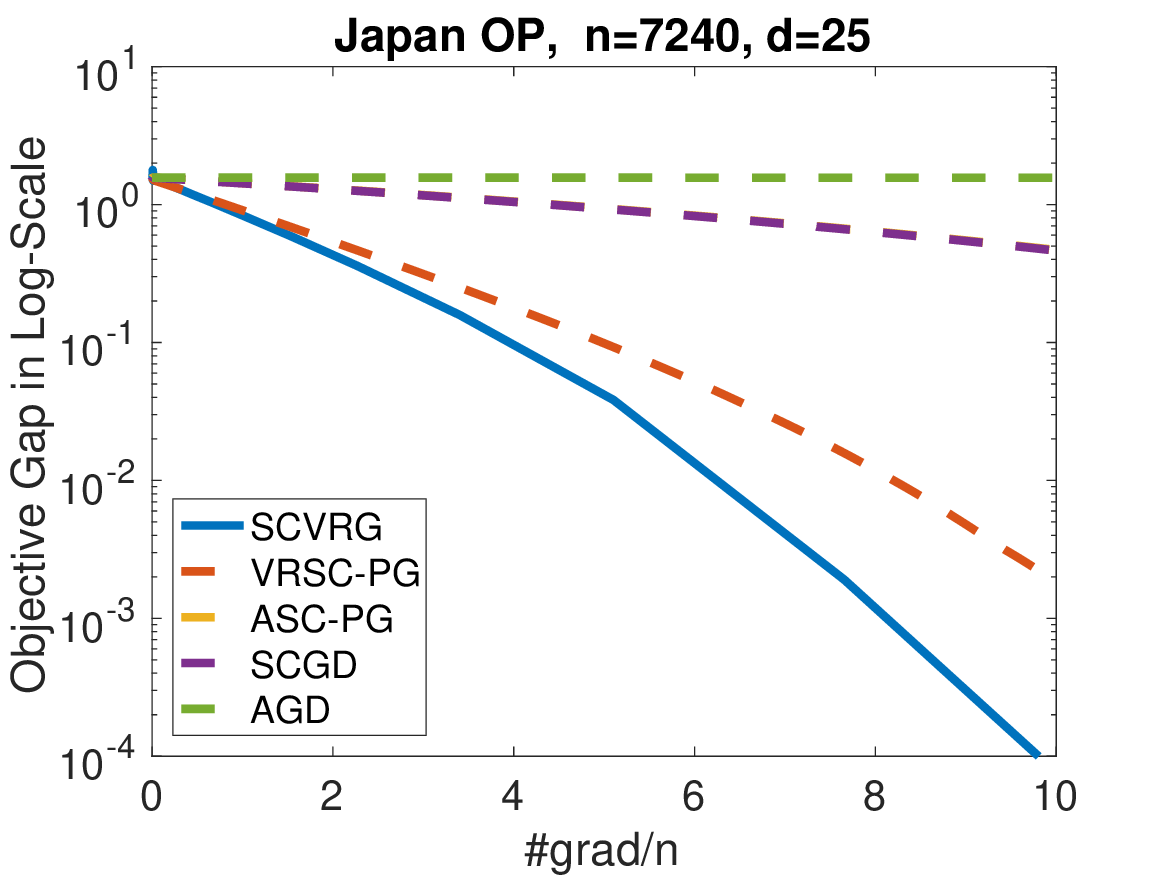}%
}%
\subfloat[America]{%
\includegraphics[width=0.2\textwidth]{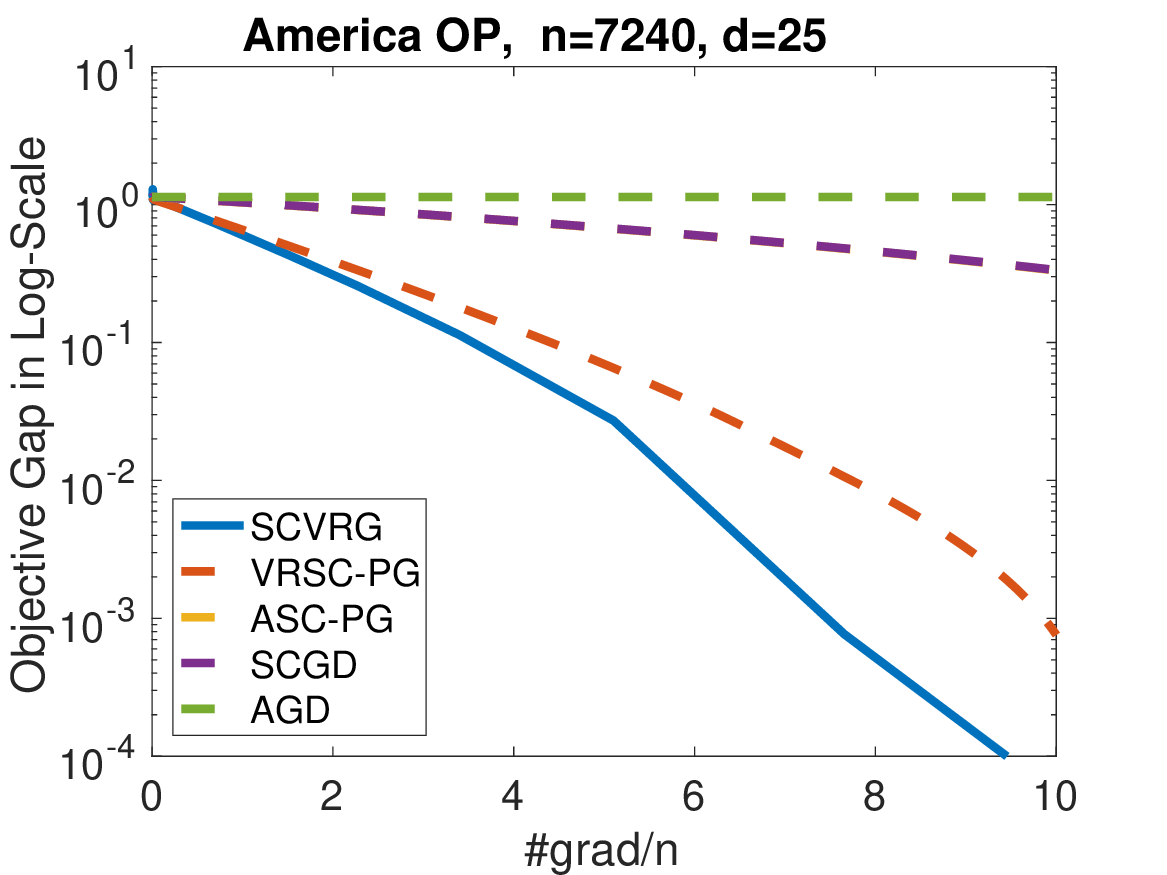}%
}%
\caption{The Performance of All Methods on five 25-Portfolio Operating Profitability Datasets} \label{Fig:Operating_Profitability}\vspace{-2em}
\end{figure*}
\begin{figure*}[!t]
\subfloat[Asia]{
\includegraphics[width=0.2\textwidth]{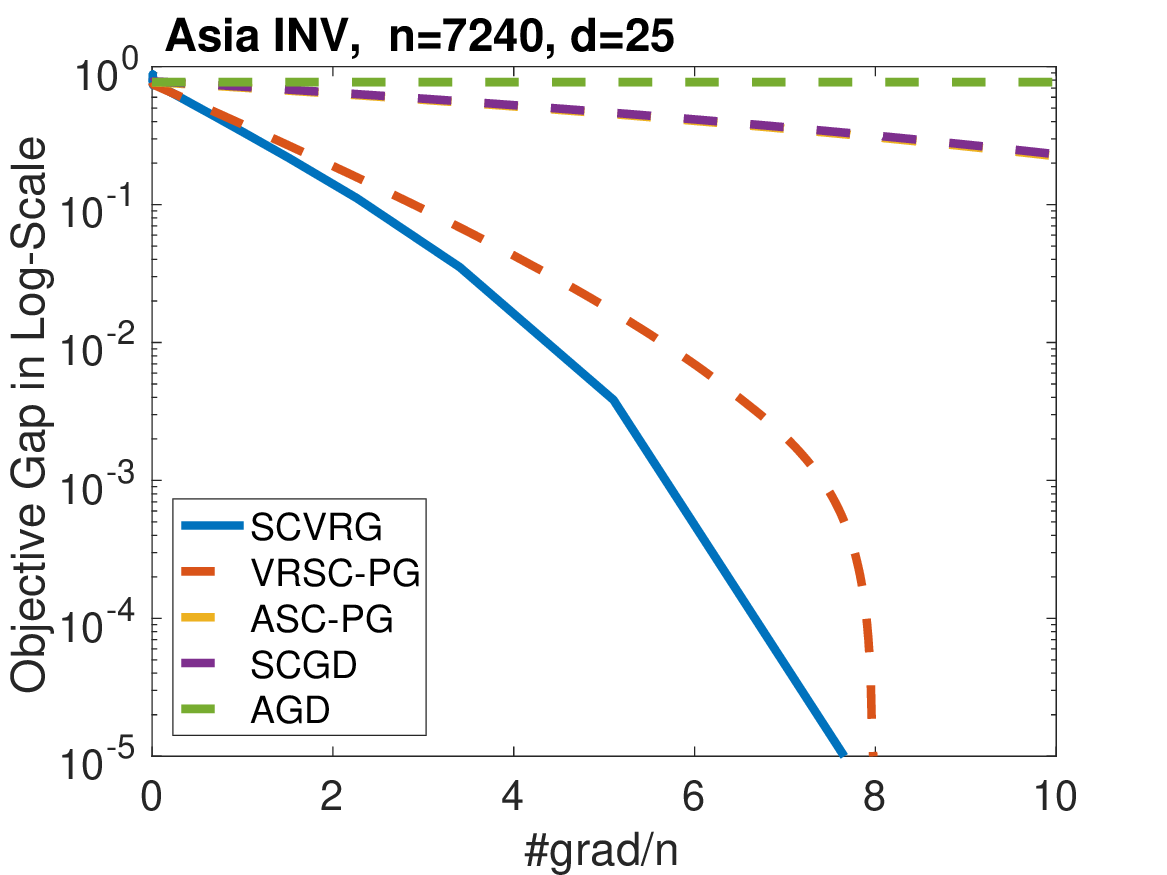}%
}
\subfloat[Europe]{%
\includegraphics[width=0.2\textwidth]{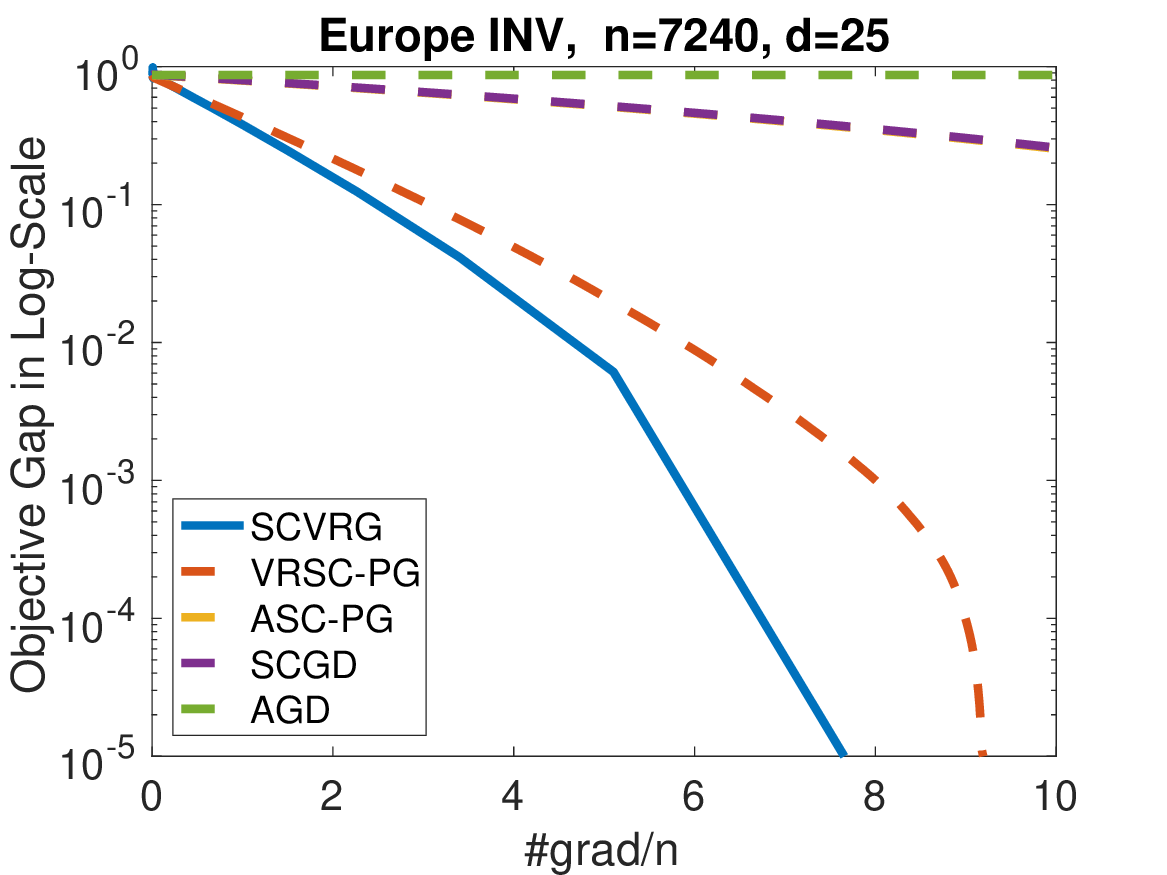}%
}%
\subfloat[Global]{%
\includegraphics[width=0.2\textwidth]{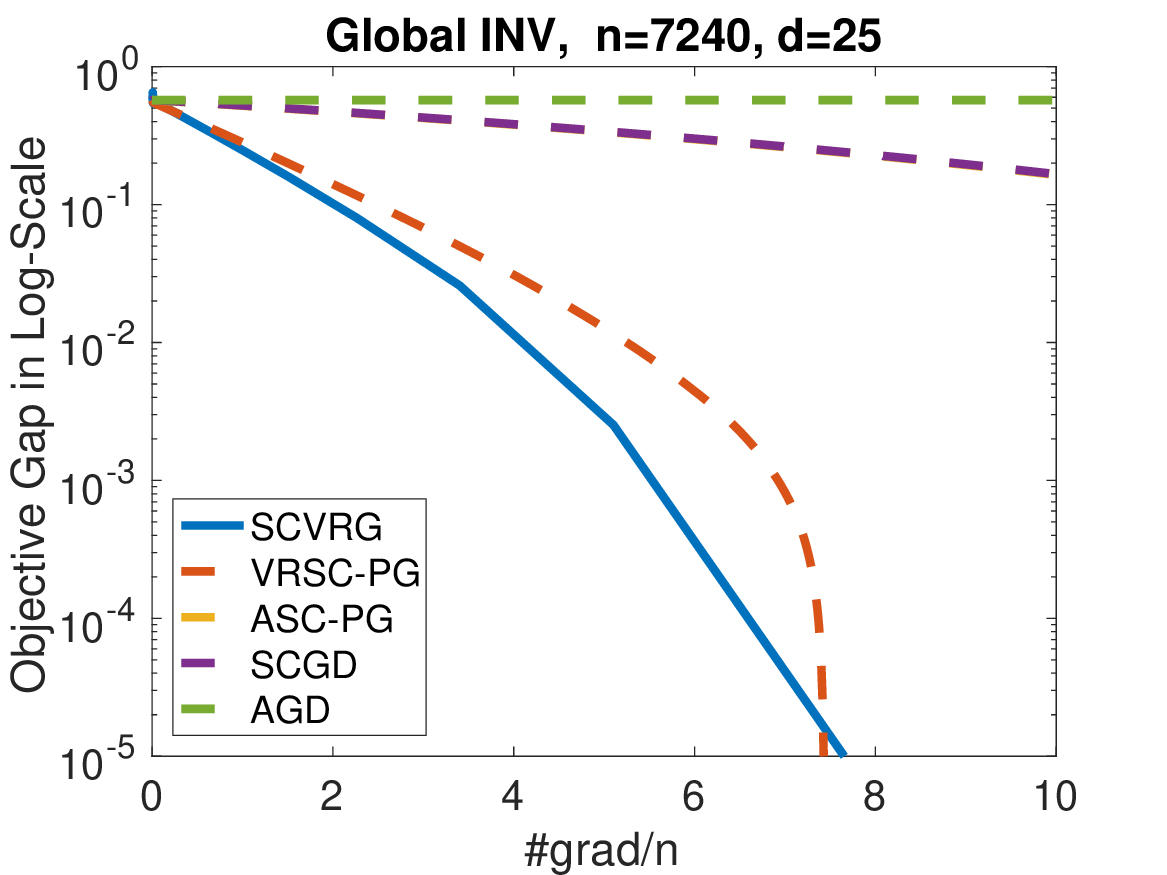}%
}%
\subfloat[Japan]{%
\includegraphics[width=0.2\textwidth]{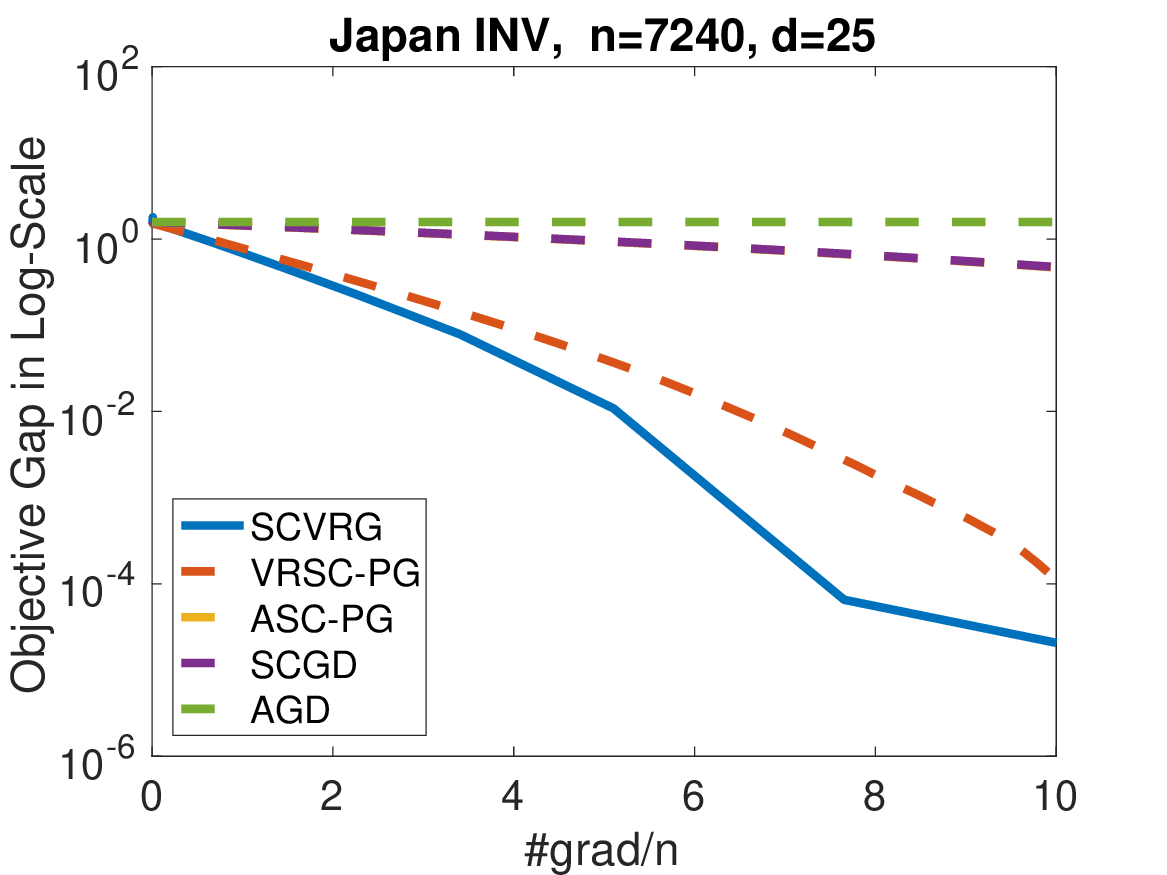}%
}%
\subfloat[America]{%
\includegraphics[width=0.2\textwidth]{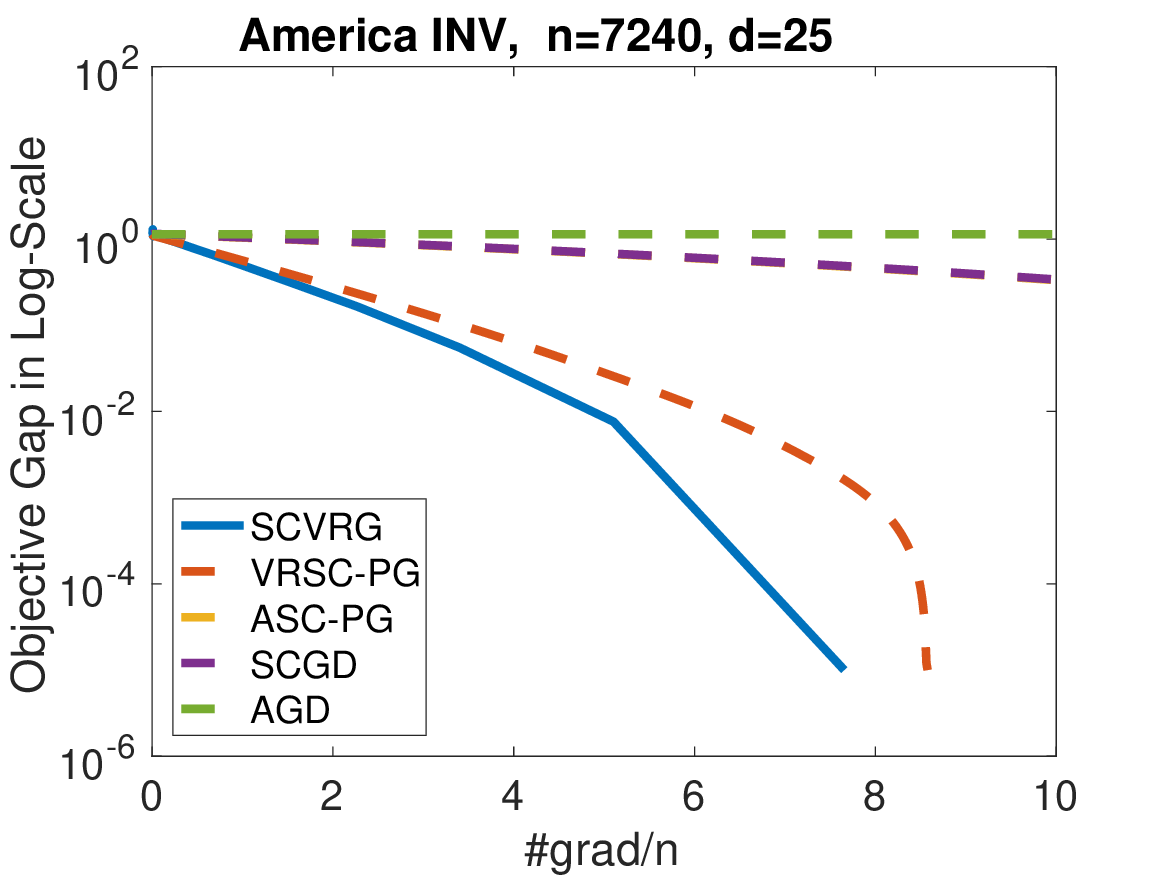}%
}%
\caption{The Performance of All Methods on five 25-Portfolio Investment Datasets}\label{Fig:Investment}\vspace{-2em}
\end{figure*}
\section{Conclusions}
We propose a stochastic compositional variance gradient algorithm for convex composition optimization with an improved sample complexity. Experiments on real-world datasets demonstrate the efficiency of our new algorithm. Future research could establish a lower bound for the sample complexity of convex composition optimization. 
\section*{APPENDIX}
We define an \textit{unbiased} estimate of $\nabla F(\x_t^{s+1})$ as
\begin{align*}
\su_t^{s+1} = & \tilde{\sv}^{s+1} + \frac{1}{b}\sum_{i_t \in \BCal_t} ([\partial g(\x_t^{s+1})]^\top\nabla f_{i_t}(g(\x_t^{s+1})) \\
& - [\tilde{\z}^{s+1}]^\top\nabla f_{i_t}(\tilde{\g}^{s+1})).  
\end{align*}
\begin{lemma}\label{Lemma:SampleGradient-Variance}
We have the following inequality, 
\begin{equation*}
\BE[\|\sv_t^{s+1} - \su_t^{s+1}\|^2 \mid \x_t^{s+1}, \tilde{\x}^s] \leq \frac{2\ell^2(\|\x_t^{s+1} - \tilde{\x}^s\|^2)}{a}. 
\end{equation*}
\end{lemma}
\begin{proof}
Using~\eqref{Update:SCVRG-v} and Cauchy-Schwarz inequality, we have
\begin{equation}\label{inequality-variance-first}
\|\sv_t^{s+1} - \su_t^{s+1}\|^2 \leq \frac{1}{b} \sum_{i_t \in \BCal_t} \|\Delta_{i_t}^{s+1}\|^2. 
\end{equation}
where $\Delta_{i_t}^{s+1} = [\z_t^{s+1}]^\top\nabla f_{i_t}(\g_t^{s+1}) - [\partial g(\x_t^{s+1})]^\top \nabla f_{i_t}(g(\x_t^{s+1}))$. Using the Cauchy-Schwarz inequality again, we have
\begin{align*}
\|\Delta_{i_t}^{s+1}\|^2 & \leq  2\|\z_t^{s+1} - \partial g(\x_t^{s+1})\|^2 \|\nabla f_{i_t}(\g_t^{s+1})\|^2 \\
& + 2\|\partial g(\x_t^{s+1})\|^2\|\nabla f_{i_t}(\g_t^{s+1}) - \nabla f_{i_t}(g(\x_t^{s+1}))\|^2.
\end{align*}
By Assumption~\ref{Assumption:Smooth-Gradient-Jacobian-Main}, we have
{\small \begin{equation}\label{inequality-variance-second}
\|\Delta_{i_t}^{s+1}\|^2 \leq 2L_f^2\|\z_t^{s+1} - \partial g(\x_t^{s+1})\|^2 + 2L_g^2\ell_f^2\|\g_t^{s+1} - g(\x_t^{s+1})\|^2.
\end{equation}}
By the definition of $\z_t^{s+1}$, we have
\begin{eqnarray*}
& & \|\z_t^{s+1} - \partial g(\x_t^{s+1})\|^2 \\
& = & \frac{1}{a^2}\|\sum_{j_t \in \ACal_t} (\partial g_{j_t}(\x_t^{s+1}) - \partial g_{j_t}(\tilde{\x}^s) - \partial g(\x_t^{s+1}) + \tilde{\z}^{s+1})\|^2.
\end{eqnarray*}
Since $\ACal_t \subseteq [m]$ are sampled uniformly with replacement with $|\ACal_t|=a$ and $\BE\|\xi-\BE\xi\|^2 \leq \BE\|\xi\|^2$, we have
{\small \begin{align}\label{inequality-variance-third}
& \BE[\|\z_t^{s+1} - \partial g(\x_t^{s+1})\|^2 \mid \x_t^{s+1}, \tilde{\x}^s] \\
= & \frac{1}{a^2}\sum_{j_t \in \ACal_t} \BE[\|\partial g_{j_t}(\x_t^{s+1}) - \partial g_{j_t}(\tilde{\x}^s) - \partial g(\x_t^{s+1}) + \tilde{\z}^{s+1}\|^2 \mid \x_t^{s+1}, \tilde{\x}^s] \nonumber \\
\leq & \frac{1}{a^2}\sum_{j_t \in \ACal_t} \BE[\|\partial g_{j_t}(\x_t^{s+1}) - \partial g_{j_t}(\tilde{\x}^s)\|^2 \mid \x_t^{s+1}, \tilde{\x}^s] \leq \frac{\ell_g^2}{a}\left\|\x_t^{s+1} - \tilde{\x}^s\right\|^2. \nonumber
\end{align}}
where the last inequality holds due to Assumption~\ref{Assumption:Smooth-Gradient-Jacobian-Main}. By the same argument, we have
\begin{equation}\label{inequality-variance-fourth}
\BE[\|\g_t^{s+1} - g(\x_t^{s+1})\|^2 \mid \x_t^{s+1}, \tilde{\x}^s] \leq \frac{L_g^2}{a}\|\x_t^{s+1} - \tilde{\x}^s\|^2. 
\end{equation}
Combining~\eqref{inequality-variance-third} and~\eqref{inequality-variance-fourth} with~\eqref{inequality-variance-second} yields that 
\begin{equation*}
\BE[\|\Delta_{i_t}^{s+1}\|^2 \mid \x_t^{s+1}, \tilde{\x}^s] \leq \left(\frac{2L_f^2\ell_g^2 + 2L_g^4\ell_f^2}{a}\right)\|\x_t^{s+1} - \tilde{\x}^s\|^2. 
\end{equation*}
Plugging the above inequality into~\eqref{inequality-variance-first} together with the definition of $\ell$ implies the desired inequality. 
\end{proof}
\begin{lemma}\label{Lemma:Gradient-Objective}
Let $\x^* \in X$ be an optimal solution, we have
\begin{align*}
\lefteqn{\BE[\|\su_t^{s+1} - \nabla F(\x_t^{s+1})\|^2 \mid \x_t^{s+1}, \tilde{\x}^s]} \\
& \leq  \frac{16\ell(\Phi(\x_t^{s+1}) - \Phi(\x^*) + \Phi(\tilde{\x}^s) - \Phi(\x^*))}{b} \\
& + \ \frac{12\ell^2(\|\x_t^{s+1} - \x^*\|^2 + \|\tilde{\x}^s - \x^*\|^2)}{b}.  
\end{align*}
\end{lemma}
\begin{proof}
By the definition of $\su_t^{s+1}$ and using the Cauchy-Schwarz inequality, we have
\begin{equation*}
\|\su_t^{s+1} - \nabla F(\x_t^{s+1})\|^2 = \frac{1}{b^2}\|\sum_{i_t \in \BCal_t} \widehat{\Delta}_{i_t}^{s+1} - \BE[\widehat{\Delta}_{i_t}^{s+1} \mid \x_t^{s+1}, \tilde{\x}^s]\|^2. 
\end{equation*}
where $\widehat{\Delta}_{i_t}^{s+1}$ is defined by
\begin{equation*}
\widehat{\Delta}_{i_t}^{s+1} = [\partial g(\x_t^{s+1})]^\top\nabla f_{i_t}(g(\x_t^{s+1})) - [\tilde{\z}^{s+1}]^\top \nabla f_{i_t}(\tilde{\g}^{s+1}). 
\end{equation*}
Since $\BCal_t \subseteq [n]$ are sampled uniformly with replacement with $|\BCal_t|=b$ and $\BE\left\|\xi-\BE\left[\xi\right]\right\|^2 \leq \BE\left\|\xi\right\|^2$, we have
\begin{align}\label{inequality-objective-first}
\lefteqn{\BE[\|\su_t^{s+1} - \nabla F(\x_t^{s+1})\|^2 \mid \x_t^{s+1}, \tilde{\x}^s]} \\
& = \frac{1}{b^2}\sum_{i_t \in \BCal_t} \BE[\|\widehat{\Delta}_{i_t}^{s+1} - \BE[\widehat{\Delta}_{i_t}^{s+1} \mid \x_t^{s+1}, \tilde{\x}^s]\|^2 \mid \x_t^{s+1}, \tilde{\x}^s] \nonumber \\
& \leq \frac{1}{b^2}\sum_{i_t \in \BCal_t} \BE[\|\widehat{\Delta}_{i_t}^{s+1}\|^2 \mid \x_t^{s+1}, \tilde{\x}^s] \nonumber
\end{align}
Using the Cauchy-Schwarz inequality again, we have
\begin{align*}
\|\widehat{\Delta}_{i_t}^{s+1}\|^2 \leq & 2\|[\partial g(\x^*)]^\top\nabla f_{i_t}(g(\x^*)) - [\tilde{\z}^{s+1}]^\top \nabla f_{i_t}(\tilde{\g}^{s+1})\|^2 \\
& \hspace{-3em} + 2\|[\partial g(\x_t^{s+1})]^\top \nabla f_{i_t}(g(\x_t^{s+1})) - [\partial g(\x^*)]^\top \nabla f_{i_t}(g(\x^*))\|^2.
\end{align*}
We define $\varphi_i(\x) = f_i(g(\x)) - \langle [\partial g(\x^*)]^\top\nabla f_i(g(\x^*)), \x - \x^*\rangle + \frac{\ell}{2}\|\x - \x^*\|^2$ for any $\x \in X$. Since $\varphi_i$ is convex and $2\ell$-gradient Lipschitz with a minimizer $\x^*$, we obtain, by applying~\cite[Theorem~2.1.5]{Nesterov-2013-Introductory} for $\varphi_i$, 
\begin{align*}
\lefteqn{ \|[\partial g(\x)]^\top\nabla f_i(g(\x)) -[\partial g(\x^*)]^\top\nabla f_i(g(\x^*)) \|^2} \\
& \leq 2\left\|\nabla\varphi_i(\x)\right\|^2 + 2\ell^2\|\x - \x^*\|^2 \\
& \leq 8\ell(\varphi_i(\x)  - \varphi_i(\x^*)) + 2\ell^2\|\x - \x^*\|^2.
\end{align*}
Putting these pieces together yields that 
\begin{align*}
\lefteqn{ \BE[\|\widehat{\Delta}_{i_t}^{s+1}\|^2 \mid \x_t^{s+1}, \tilde{\x}^s]} \\
& \leq  16\ell(F(\x_t^{s+1}) - F(\x^*) - \langle \nabla F(\x^*), \x_t^{s+1} - \x^*\rangle) \\
& \ + 16\ell(F(\tilde{\x}^s) - F(\x^*) - \langle \nabla F(\x^*), \tilde{\x}^s - \x^*\rangle) \\ 
& \ + 12\ell^2(\|\x_t^{s+1} - \x^*\|^2 + \|\tilde{\x}^s - \x^*\|^2). 
\end{align*}
Since $\x^* \in X$ is an optimal solution of $F+r$ with $F$ and $r$ being both convex, we have
\begin{equation*}
r(\x) - r(\x^*) + \left\langle\nabla F(\x^*), \x-\x^*\right\rangle \geq 0. 
\end{equation*}
which implies that 
\begin{align*}
\BE[\|\widehat{\Delta}_{i_t}^{s+1}\|^2 \mid \x_t^{s+1}, \tilde{\x}^s] \leq & 12\ell^2(\|\x_t^{s+1} - \x^*\|^2 + \|\tilde{\x}^s - \x^*\|^2) \\
& \hspace{-3em} + 16\ell(\Phi(\x_t^{s+1}) - \Phi(\x^*) + \Phi(\tilde{\x}^s) - \Phi(\x^*)). 
\end{align*}
Combining this with~\eqref{inequality-objective-first} yields the desired inequality. 
\end{proof}
Combining Lemma~\ref{Lemma:SampleGradient-Variance} and~\ref{Lemma:Gradient-Objective} with Cauchy-Schwarz inequality implies
\begin{align*}
\lefteqn{\BE[\|\sv_t^{s+1} - \nabla F(\x_t^{s+1})\|^2 \mid \x_t^{s+1}, \tilde{\x}^s]} \\
& = \BE[\|\sv_t^{s+1} - \su_t^{s+1}\|^2 + \|\su_t^{s+1} - \nabla F(\x_t^{s+1})\|^2 \mid \x_t^{s+1}, \tilde{\x}^s] \\
& \leq \frac{16\ell}{b}(\Phi(\x_t^{s+1}) - \Phi(\x^*) + \Phi(\tilde{\x}^s) - \Phi(\x^*))  \\
& \ + \left(\frac{4\ell^2}{a} + \frac{12\ell^2}{b}\right)(\|\x_t^{s+1} - \x^*\|^2 + \|\tilde{\x}^s - \x^*\|^2). 
\end{align*}
\begin{lemma}\label{Lemma:Objective-Variational-Inequality}
Let $\x^*\in X$ be an optimal solution, we have
\begin{align*}
\lefteqn{\Phi(\x^*) - \BE[\Phi(\x_{t+1}^{s+1}) \mid \x_t^{s+1}, \tilde{\x}^s]} \\
& \geq  \frac{1}{2\eta_{t+1}^{s+1}}(\BE[\|\x^* - \x_{t+1}^{s+1}\|^2 \mid \x_t^{s+1}, \tilde{\x}^s]-\|\x^* - \x_t^{s+1}\|^2) \\
& \ -\frac{\beta\ell\|\x^* - \x_t^{s+1}\|^2}{2} - \frac{\BE[\|\sv_t^{s+1} - \su_t^{s+1}\| \mid \x_t^{s+1}, \tilde{\x}^s]}{2\beta\ell} \\
& \ -\frac{\eta\BE[\|\sv_t^{s+1} - \nabla F(\x_t^{s+1})\|^2 \mid \x_t^{s+1}, \tilde{\x}^s]}{2(1-\eta\ell)}, \quad \forall \beta \in (0, 1). 
\end{align*}
\end{lemma}
\begin{proof}
By the update of $\x_{t+1}^{s+1}$ in~\eqref{Update:SCVRG-main}, we have
{\small \begin{equation*}
r(\x^*) - r(\x_{t+1}^{s+1}) + \langle\x^* - \x_{t+1}^{s+1}, \sv_t^{s+1}\rangle + \frac{\langle \x^* - \x_{t+1}^{s+1}, \x_{t+1}^{s+1} - \x_t^{s+1}\rangle}{\eta_{t+1}^{s+1}} \geq 0. 
\end{equation*}}
Furthermore, we have
\begin{align*}
\lefteqn{\langle\x^* - \x_{t+1}^{s+1}, \sv_t^{s+1}\rangle} \\
& = \langle\x^* - \x_t^{s+1}, \sv_t^{s+1}\rangle + \langle\x_t^{s+1} - \x_{t+1}^{s+1}, \nabla F(\x_t^{s+1})\rangle \\ 
& \ + \langle\x_t^{s+1} - \x_{t+1}^{s+1}, \sv_t^{s+1} - \nabla F(\x_t^{s+1})\rangle \\
& \leq  \langle\x^* - \x_t^{s+1}, \sv_t^{s+1}\rangle + F(\x_t^{s+1}) - F(\x_{t+1}^{s+1}) + \frac{\ell}{2}\|\x_t^{s+1} - \x_{t+1}^{s+1}\|^2 \\
& \ + \frac{(1-\eta_{t+1}^{s+1}\ell)\|\x_t^{s+1} - \x_{t+1}^{s+1}\|^2}{2\eta_{t+1}^{s+1}} + \frac{\eta_{t+1}^{s+1}\|\sv_t^{s+1} - \nabla F(\x_t^{s+1})\|^2}{2(1-\eta_{t+1}^{s+1}\ell)}, 
\end{align*}
where the inequality comes from Assumption~\ref{Assumption:Smooth-Gradient-Jacobian-Main} and Young's inequality. In addition, we have 
\begin{align*}
\lefteqn{\langle \x^* - \x_{t+1}^{s+1}, \x_{t+1}^{s+1} - \x_t^{s+1}\rangle} \\
& = \frac{1}{2}(\|\x^*-\x_t^{s+1}\|^2 - \|\x^* - \x_{t+1}^{s+1}\|^2)-\frac{1}{2}\|\x_{t+1}^{s+1} - \x_t^{s+1}\|^2.   
\end{align*}
Putting these pieces together yields 
{\small \begin{align}\label{inequality-variational-first}
\lefteqn{r(\x^*) + F(\x_t^{s+1}) - \Phi(\x_{t+1}^{s+1}) + \langle\x^* - \x_t^{s+1}, \sv_t^{s+1}\rangle} \\
& \geq \frac{\|\x^*-\x_{t+1}^{s+1}\|^2 - \|\x^* - \x_t^{s+1}\|^2}{2\eta_{t+1}^{s+1}} - \frac{\eta_{t+1}^{s+1}\|\sv_t^{s+1} - \nabla F(\x_t^{s+1})\|^2}{2(1-\eta_{t+1}^{s+1}\ell)}. \nonumber
\end{align}}
Since $\su_t^{s+1}$ is unbiased estimate of $\nabla F(\x_t^{s+1})$, we have
{\small \begin{align*}
\lefteqn{\BE[\langle\x^* - \x_t^{s+1}, \sv_t^{s+1} \rangle \mid \x_t^{s+1}, \tilde{\x}^s]} \\
& \leq \langle\x^* - \x_t^{s+1}, \nabla F(\x_t^{s+1})\rangle + \BE[\langle\x^* - \x_t^{s+1}, \sv_t^{s+1} - \su_t^{s+1}\rangle \mid \x_t^{s+1}, \tilde{\x}^s]. 
\end{align*}}
Using the convexity of $F$ and Young's inequality, we have
{\small \begin{align}\label{inequality-variational-second}
\lefteqn{\BE[\langle\x^* - \x_t^{s+1}, \sv_t^{s+1} \rangle \mid \x_t^{s+1}, \tilde{\x}^s], \quad \forall \beta > 0} \\ 
& \leq F(\x^*) - F(\x_t^{s+1}) + \frac{\beta\ell\|\x^*-\x_t^{s+1}\|^2}{2} + \frac{\BE[\|\sv_t^{s+1} - \su_t^{s+1}\| \mid \x_t^{s+1}, \tilde{\x}^s]}{2\beta\ell}. \nonumber
\end{align}}
Since $\eta_{t+1}^{s+1} \leq \eta$, we have $\frac{\eta_{t+1}^{s+1}}{2(1-\eta_{t+1}^{s+1}\ell)} \leq \frac{\eta}{2(1-\eta\ell)}$. 
Taking the conditional expectation of~\eqref{inequality-variational-first} together with~\eqref{inequality-variational-second} and the above inequality yields the desired inequality. 
\end{proof}
\begin{lemma}\label{Lemma:Objective-Gap}
Let $\beta \in (0, 1)$, $k_0\geq 1$ and the sample sizes satisfy that $a \geq 2\ell^2/\beta^2$ and $b \geq \ell^2/\beta^2$. The parameter $\eta > 0$ satisfy that $\eta \leq \min\{1/30\beta T \ell, 1/25\ell\}$. Then we have
\begin{align*}
\BE[\Phi(\tilde{\x}^{s+1}) - \Phi(\x^*)] \leq & \frac{\Phi(\x^0)-\Phi(\x^*)}{2^{S-1}} + \frac{5\beta\ell\|\x^*-\x^0\|^2}{2^S} \\
& + \frac{\|\x^*-\x^0\|^2}{2^S\eta k_0}. 
\end{align*}
\end{lemma}
\begin{proof}
Combining Lemma~\ref{Lemma:Objective-Variational-Inequality} with Lemma~\ref{Lemma:SampleGradient-Variance} and~\ref{Lemma:Gradient-Objective} yields that 
{\small \begin{align*}
& \Phi(\x^*) - \BE[\Phi(\x_{t+1}^{s+1}) \mid \x_t^{s+1}, \tilde{\x}^s] \\
\geq & \frac{1}{2\eta_{t+1}^{s+1}}(\BE[\|\x^* - \x_{t+1}^{s+1}\|^2 \mid \x_t^{s+1}, \tilde{\x}^s]-\|\x^* - \x_t^{s+1}\|^2) -\frac{\beta\ell\|\x^* - \x_t^{s+1}\|^2}{2} \\
& - \left(\frac{2\ell}{a\beta} + \frac{2\eta\ell^2}{a(1-\eta\ell)} +  \frac{6\eta\ell^2}{b(1-\eta\ell)}\right)(\|\x_t^{s+1} - \x^*\|^2 + \|\tilde{\x}^s - \x^*\|^2) \\
& -\frac{8\eta\ell}{b(1-\eta\ell)}(\Phi(\x_t^{s+1}) - \Phi(\x^*) + \Phi(\tilde{\x}^s) - \Phi(\x^*)), \quad \forall \beta \in (0, 1).
\end{align*}}
Since $a \geq 2/\beta^2$, $b \geq 1/\beta^2$ and $0 < \eta \leq 1/25\ell$, we have
\begin{equation*}
\frac{2\ell}{a\beta} \leq \beta\ell, \quad \frac{8\eta\ell}{b(1-\eta\ell)} \leq \frac{1}{3}, \quad \frac{2\eta\ell^2}{a(1-\eta\ell)} +  \frac{6\eta\ell^2}{b(1-\eta\ell)} \leq \frac{\beta^2\ell}{2}. 
\end{equation*}
Putting these pieces together with $\beta \in (0, 1)$ yields that 
{\footnotesize \begin{align}\label{inequality-gap-first}
& \Phi(\x^*) - \BE[\Phi(\x_{t+1}^{s+1}) \mid \x_t^{s+1}, \tilde{\x}^s] \\
\geq & \frac{1}{2\eta_{t+1}^{s+1}}(\BE[\|\x^* - \x_{t+1}^{s+1}\|^2 \mid \x_t^{s+1}, \tilde{\x}^s]-\|\x^* - \x_t^{s+1}\|^2) \nonumber \\
- & \frac{\beta\ell(3\|\tilde{\x}^s - \x^*\|^2 + 4\|\x_t^{s+1} - \x^*\|^2)}{2}-\frac{\Phi(\x_t^{s+1}) - \Phi(\x^*) + \Phi(\tilde{\x}^s) - \Phi(\x^*)}{3} \nonumber \\
\geq & \frac{1}{2\eta_{t+1}^{s+1}}(\BE[\|\x^* - \x_{t+1}^{s+1}\|^2 \mid \x_t^{s+1}, \tilde{\x}^s]-\|\x^* - \x_t^{s+1}\|^2) -5\beta\ell\|\x_t^{s+1} - \x^*\|^2 \nonumber \\
- & \frac{\beta\ell(3\|\tilde{\x}^s - \x^*\|^2-6\|\x_t^{s+1} - \x^*\|^2)}{2}-\frac{\Phi(\x_t^{s+1}) - \Phi(\x^*) + \Phi(\tilde{\x}^s) - \Phi(\x^*)}{3}. \nonumber 
\end{align}}
By the definition of $\eta_t^{s+1}$, we have
\begin{equation}\label{inequality-gap-second}
\frac{1}{\eta_t^{s+1}} - \frac{1}{\eta_{t+1}^{s+1}} \geq \frac{1}{2\eta\sqrt{T}\sqrt{2T}} = \frac{1}{2\sqrt{2}\eta T} \geq 10\beta\ell.  
\end{equation}
Plugging~\eqref{inequality-gap-second} into~\eqref{inequality-gap-first} yields that 
{\small \begin{align}\label{inequality-gap-third}
& \Phi(\x^*) - \BE[\Phi(\x_{t+1}^{s+1}) \mid \x_t^{s+1}, \tilde{\x}^s] \\
\geq & \frac{\BE[\|\x^* - \x_{t+1}^{s+1}\|^2 \mid \x_t^{s+1}, \tilde{\x}^s]}{2\eta_{t+1}^{s+1}}-\frac{\|\x^* - \x_t^{s+1}\|^2}{2\eta_t^{s+1}} \nonumber \\
- & \frac{\beta\ell(3\|\tilde{\x}^s - \x^*\|^2-6\|\x_t^{s+1} - \x^*\|^2)}{2} - \frac{\Phi(\x_t^{s+1}) - \Phi(\x^*) + \Phi(\tilde{\x}^s) - \Phi(\x^*)}{3}. \nonumber 
\end{align}}
Taking the expectation of~\eqref{inequality-gap-third}, summing it over $0 \leq t \leq k_{s+1}-1$ and dividing the final inequality by $k_{s+1}$ yields that
{\small \begin{align*}
& \BE\left(\sum_{t=0}^{k_{s+1}-1} \frac{\Phi(\x_{t+1}^{s+1}) + 3\beta\ell\|\x_t^{s+1} - \x^*\|^2}{k_{s+1}} - \Phi(\x^*)\right) \\
\leq & \frac{1}{3}\BE\left(\sum_{t=0}^{k_{s+1}-1} \frac{\Phi(\x_t^{s+1})}{k_{s+1}} - \Phi(\x^*) + \Phi(\tilde{\x}^s) - \Phi(\x^*)\right) + \frac{3\beta\ell}{2}\BE[\|\tilde{\x}^s-\x^*\|^2] \\
& + \frac{1}{2\eta_0^{s+1} k_{s+1}}\BE[\|\x^*-\x_0^{s+1}\|^2] - \frac{1}{2\eta_{k_{s+1}}^{s+1}k_{s+1}}\BE[\|\x^*-\x_{k_{s+1}}^{s+1}\|^2]. 
\end{align*}}
Rearranging the above inequality yields that
{\small \begin{align*}
\lefteqn{2\BE\left(\sum_{t=0}^{k_{s+1}-1} \frac{\Phi(\x_{t+1}^{s+1}) + \frac{9\beta\ell}{2}\|\x_t^{s+1} - \x^*\|^2}{k_{s+1}} - \Phi(\x^*)\right)} \\
& \leq \BE\left(\frac{3\Phi(\x_0^{s+1}) - 3\Phi(\x_{k_{s+1}}^{s+1})}{k_{s+1}} + \Phi(\tilde{\x}^s) - \Phi(\x^*)\right) + \frac{9\beta\ell}{2}\BE[\|\tilde{\x}^s-\x^*\|^2] \\
& \ + \frac{3}{2\eta_0^{s+1} k_{s+1}}\BE[\|\x^*-\x_0^{s+1}\|^2] - \frac{3}{2\eta_{k_{s+1}}^{s+1}k_{s+1}}\BE[\|\x^*-\x_{k_{s+1}}^{s+1}\|^2]. 
\end{align*}}
By the convexity of $\Phi$ and the definition of $\tilde{\x}^{s+1}$, we have
\begin{equation*}
\Phi(\tilde{\x}^{s+1}) \leq \sum_{t=0}^{k_{s+1}-1} \frac{\Phi(\x_t^{s+1})}{k_{s+1}}, \ \|\tilde{\x}^{s+1}-\x^*\|^2 \leq \sum_{t=0}^{k_{s+1}-1}\frac{\|\x_t^{s+1} - \x^*\|^2}{k_{s+1}}. 
\end{equation*}
Putting these pieces together with $\x_{k_{s+1}}^{s+1} = \x_0^{s+2}$, $\eta_{k_{s+1}}^{s+1}=\eta_{0}^{s+2}$ and $k_{s+2}=2 k_{s+1}$ yields that $E_{s+1} \leq E_s/2$ where 
\begin{align*}
E_s = & \BE[\Phi(\tilde{\x}^s) - \Phi(\x^*)] + \frac{9\beta\ell\BE[\|\x^* - \tilde{\x}^s\|^2]}{2} \\ 
& + \frac{3\BE[\|\x^* - \x_0^{s+1}\|^2]}{4\eta_{0}^{s+1} k_s} + \frac{3\BE[\Phi(\x_{0}^{s+1})-\Phi(\x^*)]}{2 k_s}. 
\end{align*}
Telescoping the inequality over $0 \leq s \leq S$ and rearranging yields the desired inequality. 
\end{proof}
\textbf{Proof of Theorem~\ref{Theorem:SCVRG-Complexity-Main}:} Let $\beta = \epsilon/90\ell D_\x^2 \in (0, 1)$, we have $k_0 = 5$ and $a \geq 1620\ell^2 D_\x^4/\epsilon^2$ and $b \geq 810\ell^2 D_\x^4/\epsilon^2$. Note that $\eta>0$ satisfies $\eta \leq 1/25\ell$ and 
\begin{equation*}
\frac{1}{30\beta T \ell} \geq \frac{3D_\x^2}{\epsilon k_0 2^S} \geq \frac{D_\x^2}{10D_\Phi + 25\ell D_\x^2} = \eta. 
\end{equation*}
Therefore, Lemma~\ref{Lemma:Objective-Gap} holds true. By the definition of $\eta$, $S$, $D_\Phi$ and $D_\x$ and the fact that $\beta \in (0, 1)$, we have
\begin{equation*}
\max\left\{\frac{\Phi(\x^0)-\Phi(\x^*)}{2^{S-1}}, \frac{3\beta\ell\|\x^*-\x^0\|^2}{2^S}, \frac{\|\x^*-\x^0\|^2}{2^S\eta k_0}\right\} \leq \frac{\epsilon}{3}.   
\end{equation*}
Therefore, we conclude that the sample complexity is 
\begin{equation*}
S(m+n) + 2^S k_0(a + b) = O\left((m+n)\log\left(\frac{1}{\epsilon}\right) + \frac{1}{\epsilon^3}\right). 
\end{equation*}
This completes the proof. 

 


\bibliographystyle{plain}
\bibliography{IEEEabrv,ref}

\begin{thebibliography}{10}

\bibitem{Allen-2016-Improved}
Z.~Allen-Zhu and Y.~Yuan.
\newblock Improved {SVRG} for non-strongly-convex or sum-of-non-convex
  objectives.
\newblock In {\em ICML}, pages 1080--1089, 2016.

\bibitem{Bertsekas-1995-Neuro}
D.~P. Bertsekas and J.~N. Tsitsiklis.
\newblock Neuro-dynamic programming: an overview.
\newblock In {\em CDC}, volume~1, pages 560--564. IEEE, 1995.

\bibitem{Huo-2018-Accelerated}
Z.~Huo, B.~Gu, J.~Liu, and H.~Huang.
\newblock Accelerated method for stochastic composition optimization with
  nonsmooth regularization.
\newblock In {\em AAAI}, 2018.

\bibitem{Lian-2017-Finite}
X.~Lian, M.~Wang, and J.~Liu.
\newblock Finite-sum composition optimization via variance reduced gradient
  descent.
\newblock In {\em AISTATS}, pages 1159--1167, 2017.

\bibitem{Murty-1987-Some}
K.~G. Murty and S.~N. Kabadi.
\newblock Some {NP}-complete problems in quadratic and nonlinear programming.
\newblock {\em Mathematical Programming}, 39(2):117--129, 1987.

\bibitem{Nesterov-2013-Introductory}
Y.~Nesterov.
\newblock {\em Introductory Lectures on Convex Optimization: A Basic Course},
  volume~87.
\newblock Springer Science \& Business Media, 2013.

\bibitem{Parikh-2014-Proximal}
N.~Parikh and S.~Boyd.
\newblock Proximal algorithms.
\newblock {\em Foundations and Trends{\textregistered} in Optimization},
  1(3):127--239, 2014.

\bibitem{Ravikumar-2009-Sparse}
P.~Ravikumar, J.~Lafferty, H.~Liu, and L.~Wasserman.
\newblock Sparse additive models.
\newblock {\em Journal of the Royal Statistical Society. Series B, Statistical
  Methodology}, pages 1009--1030, 2009.

\bibitem{Shalev-2016-Sdca}
S.~Shalev-Shwartz.
\newblock {SDCA} without duality, regularization, and individual convexity.
\newblock In {\em ICML}, pages 747--754, 2016.

\bibitem{Shapiro-2009-Lectures}
A.~Shapiro, D.~Dentcheva, and A.~Ruszczy{\'n}ski.
\newblock {\em Lectures on Stochastic Programming: Modeling and Theory}.
\newblock SIAM, 2009.

\bibitem{Sutton-1998-Reinforcement}
R.~S. Sutton and A.~G. Barto.
\newblock {\em Reinforcement Learning: An Introduction}, volume~1.
\newblock MIT press Cambridge, 1998.

\bibitem{Wainwright-2019-High}
M.~J. Wainwright.
\newblock {\em High-dimensional Statistics: A Non-asymptotic Viewpoint},
  volume~48.
\newblock Cambridge University Press, 2019.

\bibitem{Wang-2017-Stochastic}
M.~Wang, E.~X. Fang, and H.~Liu.
\newblock Stochastic compositional gradient descent: algorithms for minimizing
  compositions of expected-value functions.
\newblock {\em Mathematical Programming}, 161(1-2):419--449, 2017.

\bibitem{Wang-2017-Accelerating}
M.~Wang, J.~Liu, and E.~X. Fang.
\newblock Accelerating stochastic composition optimization.
\newblock {\em Journal of Machine Learning Research}, 18:1--23, 2017.

\bibitem{Yu-2017-Fast}
Y.~Yu and L.~Huang.
\newblock Fast stochastic variance reduced admm for stochastic composition
  optimization.
\newblock In {\em IJCAI}, pages 3364--3370. AAAI Press, 2017.

\end{thebibliography}

\end{document}